\documentclass[a4paper]{siamltex}

\usepackage{amsmath}
\usepackage{amssymb}

\usepackage{algorithm}

\usepackage{graphicx}
\usepackage{psfrag}
\usepackage{subfigure}

\usepackage{booktabs}

\usepackage{pythonhighlight}

\usepackage{color}

\DeclareMathOperator{\Div}{div}
\DeclareMathOperator{\Curl}{curl}
\DeclareMathOperator{\Grad}{grad}

\DeclareMathOperator{\tr}{tr}

\newcommand{\R}{\mathbb{R}}

\newcommand{\N}{\mathbb{N}}
\newcommand{\Poly}[1]{\mathcal{P}^{#1}}

\newcommand{\testspace}{\hat{V}}
\newcommand{\trialspace}{V}
\newcommand{\dualtestspace}{\hat{V}^*}
\newcommand{\dualtrialspace}{V^*}

\newcommand{\EVh}{W_h}

\newcommand{\goal}{\mathcal{M}}
\newcommand{\tol}{\epsilon}
\newcommand{\facet}{S}
\newcommand{\bubble}{b}
\newcommand{\cone}{\beta}

\newcommand{\dx}{\,\mathrm{d}x}
\newcommand{\ds}{\,\mathrm{d}s}

\newcommand{\dolfin}{\textrm{DOLFIN}}
\newcommand{\fenics}{\textrm{FEniCS}}
\newcommand{\ufl}{\textrm{UFL}}
\newcommand{\ffc}{\textrm{FFC}}

\newcommand{\triang}{\mathcal{T}}
\newcommand{\foralls}{\forall \,}
\newcommand{\assumption}[1]{A#1}
\newcommand{\inner}[2]{\langle #1, #2 \rangle}
\newcommand{\average}[1]{[#1]}
\newcommand{\emp}[1]{\texttt{#1}}

\newtheorem{thm}{Theorem}[section]
\newtheorem{lem}[thm]{Lemma}

\numberwithin{equation}{section}

\title{Automated goal-oriented error control I: Stationary variational problems}

\author{Marie E. Rognes%
 \thanks{Center for Biomedical Computing at Simula Research
   Laboratory, P.O. Box 134, 1325 Lysaker, Norway (meg@simula.no,
   logg@simula.no). This work is supported by an Outstanding Young
   Investigator grant from the Research Council of Norway, NFR
   180450. This work is also supported by a Center of Excellence grant
   from the Research Council of Norway to the Center for Biomedical
   Computing at Simula Research Laboratory.}
\and Anders Logg\footnotemark[1]
\thanks{Department of Informatics, University of Oslo, Norway.}}

\begin{document}

\maketitle

\renewcommand{\thefootnote}{\arabic{footnote}}

%------------------------------------------------------------------------------

\begin{abstract}
  This article presents a general and novel approach to the automation
  of goal-oriented error control in the solution of nonlinear
  stationary finite element variational problems. The approach is
  based on automated linearization to obtain the linearized dual
  problem, automated derivation and evaluation of a~posteriori error
  estimates, and automated adaptive mesh refinement to control the
  error in a given goal functional to within a given
  tolerance. Numerical examples representing a variety of different
  discretizations of linear and nonlinear partial differential
  equations are presented, including Poisson's equation, a mixed
  formulation of linear elasticity, and the incompressible
  Navier--Stokes equations.
\end{abstract}

\begin{keywords}
  finite element method, a posteriori, error control, dual problem,
  adaptivity, automation, nonlinear
\end{keywords}

\begin{AMS}
65N30, 68N30
\end{AMS}

%------------------------------------------------------------------------------
\section{Introduction}

For any numerical method, it is of critical importance that the
accuracy of computed solutions may be assessed. For the numerical
solution of differential equations, the accuracy is typically assessed
manually by computing a sequence of solutions using successive
refinement until it is judged that the solution has
``converged''. This approach is unreliable as well as
time-consuming. It may also be impossible, since computing resources
may be exhausted long before convergence has been reached.

For finite element discretizations, classic \emph{a~posteriori} error
analysis provides a framework for controlling the approximation error
measured in some Sobolev norm,
cf.~\cite{ainsworth_posteriori_2000}. Over the last two decades,
\emph{goal-oriented} error control has been developed as an extension
of the classic \emph{a~posteriori}
analysis~\cite{becker_optimal_2001,eriksson_introduction_1995}. The
problem of goal-oriented error control for stationary variational
problems can be posed as follows. Consider the following canonical
variational problem: find $u \in \trialspace$ such that
\begin{equation}
  \label{eq:primal}
  F(u; v) = 0 \quad \foralls v \in \testspace,
\end{equation}
where $F : \trialspace \times \testspace \rightarrow \R$ is a
semilinear form (linear in $v$) on a pair of trial and test
spaces~$(\trialspace, \testspace)$. A goal-oriented adaptive algorithm
seeks to find an approximate solution $u_h \approx u$
of~\eqref{eq:primal} such that
\begin{equation*}
  \eta = |\goal(u) - \goal(u_h)| \leq \tol,
\end{equation*}
where $\goal : \trialspace \rightarrow \R$ is a given goal functional,
$\tol > 0$ is a given tolerance, and $\eta$ is here defined as the
error in the given goal functional. In other words, goal-oriented
error control allows the construction of an adaptive algorithm that
targets a simulation to the efficient computation of a specific
quantity of interest.

The framework developed
in~\cite{becker_optimal_2001,eriksson_introduction_1995} provides a
general method for deriving an \emph{a~posteriori} estimate of the
error and adaptive refinement indicators, based on the solution of an
auxiliary linearized adjoint (dual) problem. This framework is
directly applicable to a large class of finite element variational
problems. However, a certain level of expertise is required to derive
the error estimate for a particular problem and to implement the
corresponding adaptive solver. In particular, the derivation of the
dual problem involves the linearization of a possibly complicated
nonlinear problem. Furthermore, both the derivation and evaluation of
the \emph{a~posteriori} error estimate remain nontrivial (at least in
practice). Moreover, each derivation must be carried out on a
problem-by-problem basis. As a result, goal-oriented error control
remains a tool for experts and usually requires a substantial effort
to implement.

In this work, we seek to automate goal-oriented error control. We
present a fully automatic approach for computation of error estimates
and adaptive refinement; that is, without any manual analysis,
preparation, or intervention. To our knowledge, no such automation has
previously been presented, and in particular not realized. The
strategy presented here requires a minimal amount of input and expert
knowledge; the only input required by our adaptive algorithm is the
semilinear form $F$, the functional~$\mathcal{M}$, and the
tolerance~$\tol$. Based on the given input, the adaptive algorithm
automatically generates the dual problem, the \emph{a~posteriori}
error estimate, and attempts to compute an approximate solution~$u_h$
that meets the given tolerance for the given functional. In
particular, problem-tuned error estimates and indicators are generated
without any manual derivations. This has the potential of rendering
state-of-the-art goal-oriented error control fully accessible to
non-experts at no additional implementational cost.

We emphasize that although one may, in principle, manually carry out
the necessary analysis and implementation in any particular case,
automation plays an important role since it (i) makes expert knowledge
accessible to non-experts, (ii) speeds up development cycles, and
(iii) enables more complex problems to be tackled which would
otherwise require considerable effort to analyze and implement. A
similar and successful automation effort as part of the FEniCS
Project~\cite{LoggMardalEtAl2012a,kirby_fiat:new_2004,kirby_compiler_2006,logg_dolfin_2009}
has led to the development of methodology and tools that automate the
discretization of a large class of partial differential equations by
the finite element method.

In this paper, we limit the discussion to stationary variational
problems. The error estimates and indicators generated by the
automated algorithm can be viewed as a version of the
dual-weighted-residual estimates of~\cite{becker_optimal_2001}. We
emphasize that the main target of this paper is to present an
automation of a method for goal-oriented error control; in particular,
it is not our intention to improve the method theoretically nor in
detail examine the cases where the method is known to work
poorly. Also, the question of automated error control for
time-dependent problems will be considered in later works.

The remainder of this introduction describes the organization of the
paper. The first two sections, Sections~\ref{sec:notation}
and~\ref{sec:linear}, establish the abstract problem setting by
defining notation and summarizing the well-established goal-oriented
error estimation framework for linear variational problems. The linear
setting is discussed first for the sake of clarity and brevity: the
extension to the nonlinear case is summarized in
Section~\ref{sec:nonlinear}. The primary novel contributions of this
paper are contained in
Sections~\ref{sec:residual_representation},~\ref{sec:dual_approximation},
and \ref{sec:adaptivity}: the main result of this paper is an
automated strategy for the computation of error estimates and
indicators. This strategy relies on two key components: first, a
procedure, applicable to a general class of stationary variational
problems, for the derivation of a strong residual representation from
a weak residual representation
(Section~\ref{sec:residual_representation}). Second, the evaluation of
the error estimates relies on a dual approximation: an affordable
strategy for obtaining an improved dual approximation by extrapolation
for a general class of finite element spaces is described in
Section~\ref{sec:dual_approximation}.

Combining the goal-oriented error estimation framework with the
automated techniques introduced in
Section~\ref{sec:residual_representation}
and~\ref{sec:dual_approximation}, yields the automated adaptive
algorithm described in Section~\ref{sec:adaptivity}. A realization of
this automated algorithm has been implemented as part of the FEniCS
Project, and is, in particular, available through both the Python and
C++ versions of the DOLFIN library. A simple example of its use and
some aspects of the implementation are also discussed in this section.
In Section~\ref{sec:numerics}, we apply the presented framework to
three examples of varying complexity: the Poisson equation, a
three-field mixed formulation for the linear elasticity equations, and
the stationary Navier--Stokes equations. Finally, we conclude and
discuss further work in Section~\ref{sec:conclusion}.

%% The outline of this paper is as follows. Our notation is introduced in
%% Section~\ref{sec:notation}. In Section~\ref{sec:linear}, we present
%% the goal-oriented error estimation framework for linear variational
%% problems. The linear setting is discussed first for the sake of
%% clarity. The extension to the nonlinear case is presented in
%% Section~\ref{sec:nonlinear}. Our approach requires an automated
%% strategy for the computation of error estimates and indicators. Such a
%% strategy is presented in
%% Section~\ref{sec:residual_representation}. Further, the evaluation of
%% the error estimates rely on a dual approximation; a strategy for
%% extrapolating an improved dual approximation is described in
%% Section~\ref{sec:dual_approximation}. The full adaptive algorithm is
%% summarized in Section~\ref{sec:adaptivity}.

%------------------------------------------------------------------------------
\section{Notation}
\label{sec:notation}

Throughout this paper, $\Omega \subset \R^d$ denotes an open, bounded
domain with boundary $\partial \Omega$. We will generally assume that
$\Omega$ is polyhedral such that it can be exactly represented by an
admissible, simplicial tessellation $\triang_h$. The boundary will
typically be the union of two disjoint parts, denoted $\partial
\Omega_D$ and $\partial \Omega_N$.

In general, the notation $V(X; Y)$ is used to denote the space of
fields $X \rightarrow Y$ with regularity properties specified by
$V$. If $Y = \R$, this argument is omitted. For $L^2(K; \R^d)$; that
is, the space of $d$-vector fields on $K \subseteq \Omega$ in which
each component is square integrable, the inner product reads $\langle
\cdot, \cdot \rangle_{K}$, and the norm is denoted $||\cdot||_{K}$. If
$K = \Omega$, the subscript is omitted. For $m = 1, 2, \dots$,
$H^m(\Omega)$ denotes the space of square integrable functions with
$m$ square integrable distributional derivatives. Also, $H^1_{g,
  \Gamma} = \{ u \in H^1(\Omega): u|_{\Gamma} = g \}$. Similarly,
$H(\Div, \Omega)$ denotes the space of square integrable vector fields
with square integrable divergence. Note that both the gradient of a
vector field and the divergence of a matrix field are applied
row-wise.

A form $a: W_1 \times \cdots \times W_{n} \times V_{\rho} \times \cdots
\times V_1 \rightarrow \R$, written $a(w_1, \dots, w_n; v_{\rho},
\dots, v_1)$, is (possibly) nonlinear in all arguments preceding
the semi-colon, but linear in all arguments following the semi-colon.

%------------------------------------------------------------------------------
\section{A framework for goal-oriented error control}
\label{sec:linear}

In this section, we present a general framework for goal-oriented
error control for conforming finite element discretizations of
stationary variational problems. The framework is a summary of the
paradigm developed
in~\cite{becker_optimal_2001,eriksson_introduction_1995}. For clarity,
we restrict our attention to linear variational problems and linear
goal functionals. Extensions to nonlinear problems and nonlinear goal
functionals are made in Section~\ref{sec:nonlinear}.

Let $\trialspace$ and $\testspace$ be Hilbert spaces of functions or
fields defined on a domain $\Omega \subset \R^d$ for $d = 1, 2, 3$. In
this section, we consider the following linear variational problem:
find $u \in \trialspace$ such that
\begin{equation}
  \label{eq:primal,linear}
  a(u, v) = L(v) \quad \foralls v \in \testspace.
\end{equation}
We assume that $a: \trialspace \times \testspace \rightarrow \R$ is a
continuous, bilinear form, and that $L: \testspace \rightarrow \R$ is
a continuous, linear form. We shall further assume that the problem is
well-posed; that is, there exists a unique solution $u$ that depends
continuously on any given data. The variational problem defined
by~\eqref{eq:primal,linear} will be referred to as the \emph{primal
  problem} and~$u$ will be referred to as the \emph{primal solution}.

Let $\triang_h$ be an admissible simplicial tessellation of $\Omega$
(to be determined) and assume that $\trialspace_h \subset \trialspace$
and $\testspace_h \subset \testspace$ are finite element spaces
defined relative to $\triang_h$. The finite element approximation
of~\eqref{eq:primal,linear} then reads: find $u_h \in \trialspace_h$
such that
\begin{equation}
  \label{eq:primal,linear,discrete}
  a(u_h, v) = L(v) \quad \foralls v \in \testspace_h.
\end{equation}
We assume that the spaces $\trialspace_h$ and $\testspace_h$ satisfy
an appropriate discrete inf--sup condition such that a unique discrete
solution exists. The problem~\eqref{eq:primal,linear,discrete} will be
referred to as the discrete primal problem and $u_h$ the discrete
primal solution.

We are interested in estimating the magnitude of the error in a given
goal functional $\goal: \trialspace \rightarrow \R$. Moreover, for a
given tolerance $\tol > 0$, we aim to find
$(\trialspace_h, \testspace_h)$ such that the corresponding finite
element approximation $u_h$, as defined
by~\eqref{eq:primal,linear,discrete}, satisfies
\begin{equation}
  \label{eq:zegoal}
  \eta \equiv |\goal(u) - \goal(u_h)| \leq \tol.
\end{equation}
In addition, we would like to compute the value of the goal functional
$\goal(u_h)$ efficiently, ideally using a minimal amount of work.

In order to estimate the magnitude of the error $\eta$, we first
define the (weak) residual relative to the approximation $u_h$,
\begin{equation}
  \label{eq:residual}
  r(v) = L(v) - a(u_h, v).
\end{equation}
Some remarks are in order. First, $r$ is a bounded, linear functional
by the continuity and linearity of $a$ and $L$. Second, as a
consequence of the Galerkin orthogonality implied by $\testspace_h
\subset \testspace$, the residual vanishes on $\testspace_h$. In other
words,
\begin{equation}
  \label{eq:galerkinorthogonality}
  r(v) = 0 \quad \foralls v \in \testspace_h.
\end{equation}

Next, we define the (weak) dual problem: find $z \in \dualtrialspace$
such that
\begin{equation}
  \label{eq:dual}
  a^{*}(z, v) = \goal(v) \quad \foralls v \in \dualtestspace,
\end{equation}
where $(\dualtrialspace, \dualtestspace)$ is the pair of dual trial
and test spaces, and $a^{*}$ denotes the adjoint of $a$; that is,
$a^{*}(v, w) = a(w, v)$. We shall assume that the dual
problem~\eqref{eq:dual} is well-posed, and that there thus exists a
\emph{dual solution} $z$ solving~\eqref{eq:dual} with continuous
dependence on the input data. Moreover, we assume that the dual trial
and test spaces are chosen such that $u - u_h \in \dualtestspace$ and
$z \in \testspace$. This holds if $\dualtestspace=\trialspace_0 = \{v
- w: v, w \in \trialspace\}$ and $\dualtrialspace=\testspace$.

Combining~\eqref{eq:dual},~\eqref{eq:residual}, and
\eqref{eq:primal,linear}, we find that
\begin{equation*}
  \goal(u) - \goal(u_h)
  = a^*(z, u - u_h)
  = a(u - u_h, z)
  = L(z) - a(u_h, z)
  \equiv r(z).
\end{equation*}
The error $\goal(u) - \goal(u_h)$ is thus equal to the (weak)
residual~$r$ evaluated at the dual solution~$z$. By the Galerkin
orthogonality~\eqref{eq:galerkinorthogonality}, we obtain the
following error representation:
\begin{equation} \label{eq:error_representation}
  \goal(u) - \goal(u_h) = r(z) = r(z - \pi_h z).
\end{equation}
Here, $\pi_h z \in \testspace_h$ is an arbitrary test space field,
typically an interpolant of the dual solution.

An identical error representation is obtained for nonlinear
variational problems and nonlinear goal functionals with a suitable
definition of the dual problem. We return to this issue in
Section~\ref{sec:nonlinear}. It follows that if one can compute (or
approximate) the solution of the dual problem, one may estimate the
size of the error by a direct evaluation of the residual. However,
some concerns remain that require special attention. First, the error
representation~\eqref{eq:error_representation} is not directly useful
as an error indicator. The derivation of an \emph{a~posteriori} error
estimate and corresponding error indicators from the error
representation has traditionally required manual analysis, typically
involving some form of integration by parts and a redistribution of
boundary terms (fluxes) over cell facets. Second,
for~\eqref{eq:error_representation} to give a useful estimate of the
size of the error, care must be taken when solving the dual
problem~\eqref{eq:dual}. In particular, the error representation
evaluates to zero if the dual solution is approximated
in~$\testspace_h$. Finally, the derivation of the dual problem may
involve the differentiation of a nonlinear variational form. We
discuss how each of these issues can be automated in the subsequent
sections.

%------------------------------------------------------------------------------
\section{Automated derivation of error estimates and error indicators}
\label{sec:residual_representation}

This section presents a novel approach to the derivation of an
\emph{a~posteriori} error estimate and corresponding error indicators for a
general class of stationary variational problems. The starting point
is the general abstract error
representation~\eqref{eq:error_representation}. The proposed generic
representation for the estimate and indicators is motivated and
introduced in Section~\ref{subsec:generic_representation}, followed by
a strategy for computing this representation in
Sections~\ref{subsec:automated_computation}--\ref{subsec:solvability}. The
key idea is the computation of a problem-tuned strong residual
representation using only ingredients from the general abstract form;
this concept constitutes one of the pillars for the complete automated
strategy. We emphasize that the resulting error estimate coincides
with classical duality-based error estimates that may be derived
manually (using integration by parts) for standard problems such as
the Poisson problem. However, we present here an approach that allows
these error estimates to be generated and evaluated automatically. We
also note that the automatically derived error indicators differ from
the standard duality-based error estimates resulting from an
integration by parts followed by one or more inequalities.

\subsection{A generic residual representation}
\label{subsec:generic_representation}

For motivational purposes, we start by considering the standard
derivation of a dual-weighted residual error estimator for Poisson's
equation: $- \Delta u = f$ and its corresponding variational problem
defined by $a(u, v) = \inner{\grad u}{\grad v}$ and $L(v) =
\inner{f}{v}$ on $\trialspace = \testspace = H^1_0(\Omega)$. By
integrating the weak residual by parts cell-wise, one obtains
\begin{equation*}
  \begin{split}
    r(z)
    &\equiv
    L(z) - a(u_h, z)
    \equiv
    \inner{f}{z} - \inner{\grad u_h}{\grad z} \\
    &=
    \sum_{T\in\triang} \inner{f}{z}_T - \inner{\grad u_h}{\grad z}_T
    =
    \sum_{T\in\triang}
    \inner{f + \Delta u_h}{z}_T +
    \inner{- \partial_n u_h}{z}_{\partial T} \\
    &=
    \sum_{T\in\triang}
    \inner{f + \Delta u_h}{z}_T +
    \inner{\average{- \partial_n u_h}}{z}_{\partial T},
  \end{split}
\end{equation*}
where $\average{- \partial_n u_h}$ denotes an appropriate
redistribution of the flux over cell facets. Several choices are
possible, see for example~\cite[Chap.~6]{ainsworth_posteriori_2000},
but we here make the simplest possible choice and distribute the flux
equally. In particular, we define $\average{\partial_n u_h}|_{\facet}
= \frac{1}{2}(\grad u_h|_{T} \cdot n + \grad u_h|_{T'} \cdot n')$ over
all internal facets $S$ shared by two cells $T$ and $T'$, and
$\average{\partial_n u_h}|_{\facet} = \partial_n u_h|_{\facet}$ on
external facets (facets on the boundary of $\Omega$). Hence, one may
estimate the error by
\begin{equation} \label{eq:errorestimate}
  |\goal(u) - \goal(u_h)| \leq \sum_{T\in\triang} \eta_T,
\end{equation}
where the error indicator $\eta_T$ is given by
\begin{equation}
  \label{eq:errorindicator:poisson}
  \eta_T =
  |\inner{f + \Delta u_h}{z - \pi_h z}_T +
  \inner{\average{- \partial_n u_h}}{z - \pi_h z}_{\partial T}|.
\end{equation}

We note that although one may in principle use $\eta_T =
|\inner{f}{z}_T - \inner{\grad u_h}{\grad z}_T|$ as an error indicator
(without integrating by parts and redistributing the normal
derivative), that indicator is much less efficient than the error
indicator defined in~\eqref{eq:errorindicator:poisson}. Both
indicators will sum up to the same value (if taken with signs), but
only as a result of cancellation. The error
indicator~\eqref{eq:errorindicator:poisson} is generally smaller in
magnitude, scales better with mesh refinement, and gives a sharper
error bound when summed without
signs. See~\cite{wahlberg_evaluation_2009} for an extended discussion.

Estimates similar to~\eqref{eq:errorestimate} have been derived by
hand (originally for use with norm-based error indicators) for a
variety of equations. A non-exhaustive list of examples (including
purely norm-based indicators) includes standard finite element
discretizations of the Poisson equation~\cite{babuska_error_1978},
various mixed formulations for the Stokes equations and stationary
Navier--Stokes equations~\cite{verfurth_posteriori_1989},
$H(\Div)$-based discretizations of the mixed Poisson and mixed
elasticity equations~\cite{braess_posteriori_1996,
  lonsing_posteriori_2004}, and $H(\Curl)$-based discretizations for
problems in electromagnetics~\cite{beck2000residual}. Duality-based
goal-oriented error estimates have been derived for a number of
applications, including ordinary differential
equations~\cite{estep1994global},
plasticity~\cite{rannacher1998posteriori}, hyperbolic
systems~\cite{larson11posteriori}, reactive compressible
flow~\cite{sandboge1999adaptive}, systems of nonlinear
reaction--diffusion
equations~\cite{estep2000estimating,sandboge1998adaptive}, eigenvalue
problems~\cite{heuveline2001posteriori}, wave
propagation~\cite{bangerth2001adaptive}, radiative
transfer~\cite{richling2001radiative}, nonlinear
elasticity~\cite{larsson2002strategies}, the incompressible
Navier--Stokes equations~\cite{becker2002optimal,hoffman2004duality},
variational multiscale problems~\cite{larson2005adaptive}, and
multiphysics problems~\cite{larson2008adaptive}.

These estimates share a common factor, namely that the error is
expressed as a sum of contributions from the cells and the facets of
the mesh. Moreover, each of these estimates has been derived manually
for the specific problem at hand. Here, we aim to demonstrate that for
a large class of variational problems, one may automatically compute
an equivalent residual representation. The representation takes the
following generic form:
\begin{equation}
  \label{eq:residual_representation}
  r(v) = \sum_{T \in \triang_h} \inner{R_T}{v}_{T}
  + \inner{R_{\partial T}}{v}_{\partial T}
  = \sum_{T \in \triang_h}
  \inner{R_T}{v}_{T}
  + \average{\inner{R_{\partial T}}{v}_{\partial T}},
\end{equation}
where
\begin{equation*}
  \average{ \inner{R_{\partial T}}{v}_{\partial T}}
  = \sum_{\facet \subset \partial T \cap \Omega}
  \frac{1}{2}  \left ( \inner{R_{\partial T}}{v|_{T}}_{\facet}
  + \inner{R_{\partial T'}}{v|_{T'}}_{\facet} \right )
  + \sum_{\facet \subset \partial T \cap \partial \Omega}
  \inner{R_{\partial T}}{v}_{\facet}.
\end{equation*}
It follows that one may use as error indicators
\begin{equation} \label{eq:error_indicator}
  \eta_T = |
  \inner{R_T}{z - \pi_h z}_T +
  \average{\inner{R_{\partial T}}{z - \pi_h z}_{\partial T}}|.
\end{equation}
In these expressions, $R_T$ denotes a residual contribution evaluated
over the domain of a cell~$T$, whereas $R_{\partial T}$ denotes a
residual contribution evaluated over a cell boundary~$\partial T$.

\subsection{Automatic computation of the residual representation}
\label{subsec:automated_computation}

We shall focus our attention on a class of residuals $r$ satisfying
the following assumptions:
\begin{description}
\item[\textbf{\assumption{1}}](Global decomposition)
  The residual is a sum of local contributions:
  \begin{equation*}
    r = \sum_{T\in\triang_h} r_T.
  \end{equation*}
  \item[\textbf{\assumption{2}}](Local decomposition) Each local
    residual $r_T$ offers a local decomposition:
    \begin{equation}
      \label{eq:decomposition}
      r_T(v)
      = \inner{R_T}{v}_{T} + \inner{R_{\partial T}}{v}_{\partial T}
      \quad \foralls v \in \testspace|_{T}.
    \end{equation}
\end{description}
We note that~\assumption{1} is satisfied if the bilinear and linear
forms $a$ and $L$ are expressed as integrals over the cells and facets
of the tessellation~$\triang_h$. We also note that~\assumption{2} is
satisfied if the variational problem~\eqref{eq:primal,linear} has been
derived by testing a partial differential equation against a test
function and (possibly) integrating by parts to move derivatives onto
the test function.

For the sake of a simplified analysis, we also make the following
assumption:
\begin{description}
  \item[\textbf{\assumption{3}}] (Polynomial representation) The
    residual contributions (or, in the case of vector or tensor
    fields, each scalar component of these) are piecewise polynomial:
    \begin{equation*}
      R_T \in \Poly{p}(T),
      \quad R_{\partial T}|_{\facet} \in \Poly{q}(\facet)
      \quad \foralls S \in \partial T \quad \foralls T \in \triang_h,
      \quad p, q \in \N.
    \end{equation*}
\end{description}
We discuss the implications of this assumption below, but note that one
of the numerical examples presented does not satisfy this assumption.

\begin{figure}
  \begin{center}
    \includegraphics[width=0.6\textwidth]{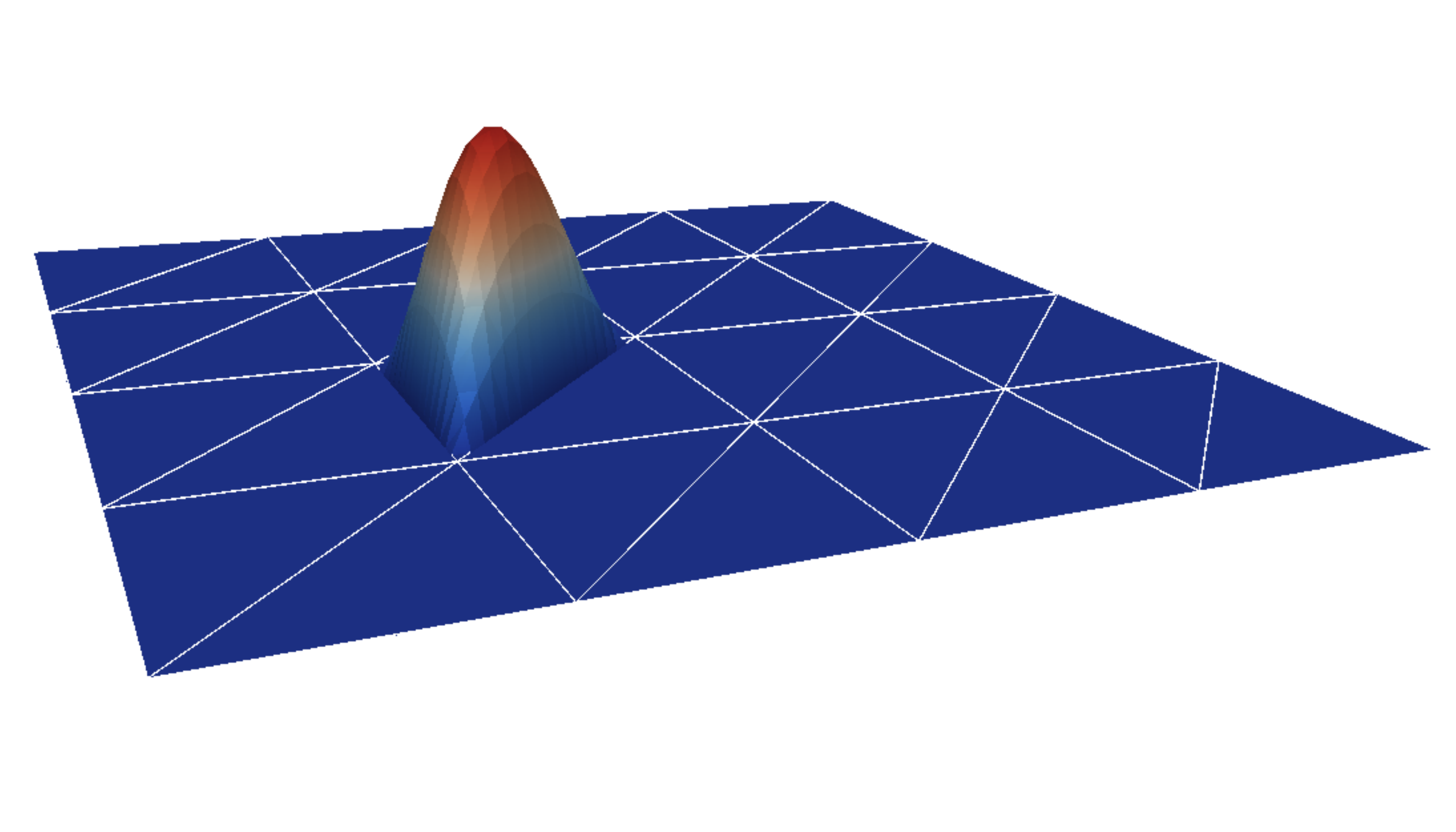}
    \includegraphics[width=0.3\textwidth]{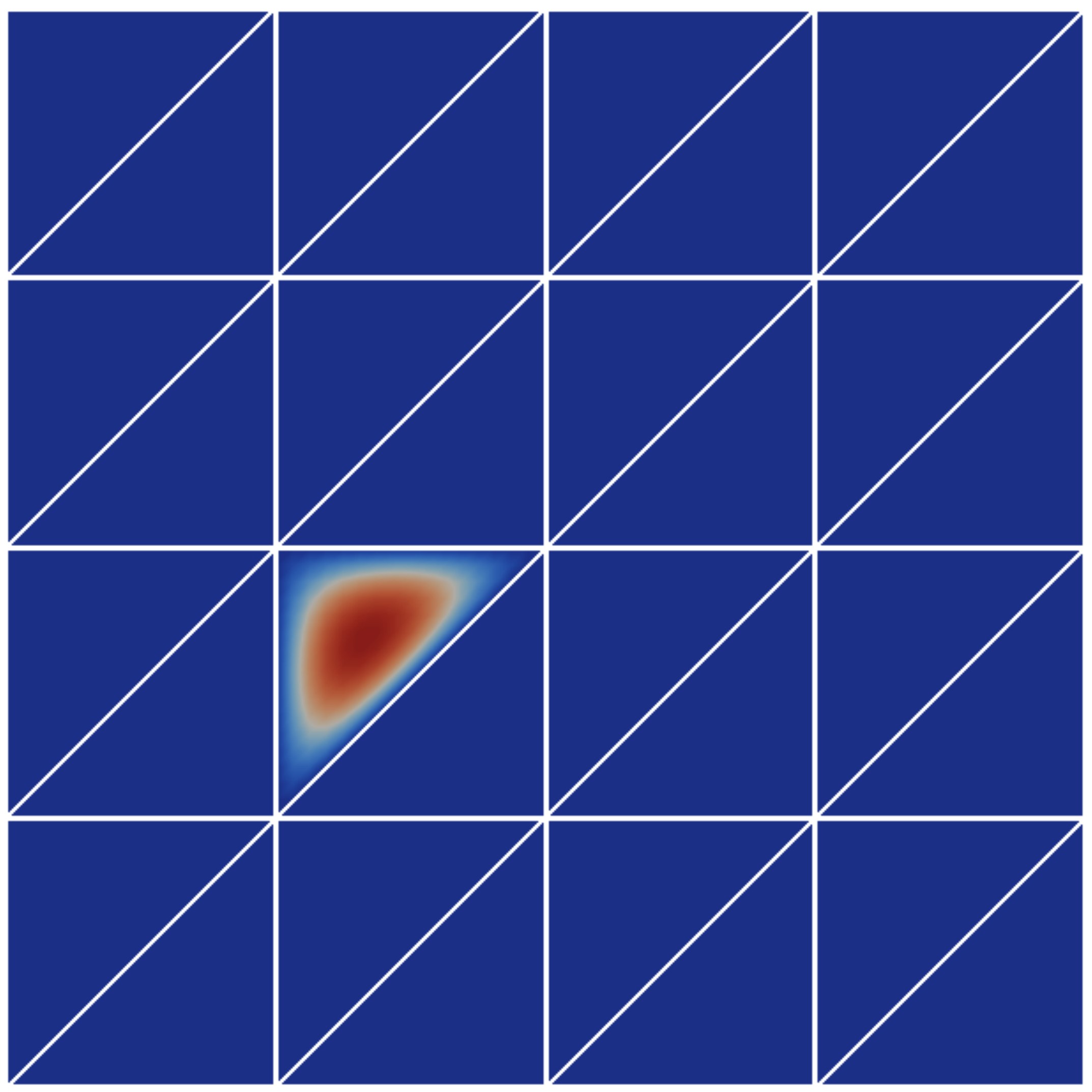}
    \caption{The bubble function~$b_T$.}
    \label{fig:cell_bubble}
  \end{center}
\end{figure}

In the following, we shall show that a residual representation may be
automatically computed if assumptions~\assumption{1}--\assumption{2}
are satisfied. More precisely, if
assumptions~\assumption{1}--\assumption{3} are satisfied, we shall
show that one may automatically compute the exact residual
representation~\eqref{eq:residual_representation} for a given
variational problem~\eqref{eq:primal,linear}. In particular, one may
directly compute the cell and facet residuals $R_T$ and $R_{\partial
T}$ by solving a set of local problems on each cell~$T$ of the
tessellation~$\triang_h$. If~\assumption{3} fails; that is, if only
assumptions~\assumption{1}--\assumption{2} are satisfied, the
automated procedure computes weighted $L^2$-projections of the
residual decomposition terms and hence an approximate residual
representation.

To compute the cell residual $R_T$, let $\{\phi_i\}_{i=1}^m$ be a
basis for $\Poly{p}(T)$ and let $\bubble_T$ denote the bubble function
on $T$. We recall that for a simplex $T \subset \R^d$, the bubble
function $\bubble_T$ is defined by
\begin{equation*}
  \bubble_{T} = \prod_{i=1}^{d+1} \lambda_{x_i}^T
\end{equation*}
where $\lambda_{x_i}^T$ is the barycentric coordinate function on $T$
associated with vertex $x_i$ (the $i$th linear Lagrange nodal basis
function on $T$). Note that $\bubble_T$ vanishes on the boundary of
$T$. See Figure~\ref{fig:cell_bubble} for an illustration. Testing the
local residual~$r_T$ against $b_T \phi_i$, we obtain the following
local problem for the cell residual $R_T$: find $R_T \in \Poly{p}(T)$
such that
\begin{equation}
  \label{eq:cell_residual}
  \inner{R_T}{\bubble_T \phi_i}_T
  = r_T(\bubble_T \phi_i), \quad i = 1, \dots m.
\end{equation}

To obtain a local problem for the facet residual~$R_{\partial T}$, we
define for each facet~$\facet$ on $T$ the \emph{cone
  function}~$\cone_{\facet}^T$ by
\begin{equation}
  \cone_{\facet}^T = \prod_{i \in I_{\facet}^T} \lambda_{x_i}^T,
\end{equation}
where $I_{\facet}^T$ is a suitably defined index set such that
$\cone_{\facet}^T|_{f} \equiv 0$ on all facets $f$ of $T$ but
$\facet$. For an illustration, see
Figure~\ref{fig:edge_bubble}. Clearly, $\cone_{\facet}^T|_{\facet} =
\bubble_{\facet}$. Next, let $\{\phi_i\}_{i=1}^{n}$ be a basis for
$\Poly{q}(T)$. Testing the local residual~$r_T$ against
$\cone_{\facet}^T \phi_i$, we obtain the following local problem for
each facet residual: find $R_{\partial T}|_{\facet} \in
\Poly{q}(\facet)$ such that
\begin{equation}
  \label{eq:facet_residual}
  \inner{R_{\partial T}|_{\facet}}{\cone_{\facet}^T \phi_i}_{\facet}
  = r_T(\cone_{\facet}^T \phi_i) - \inner{R_T}{\cone_{\facet}^T \phi_i}_T
  \quad \foralls i \in I_{\facet}^T.
\end{equation}
We prove below that by assumptions~\assumption{1}--\assumption{3}, the
local problems~\eqref{eq:cell_residual} and~\eqref{eq:facet_residual}
uniquely define the cell and facet residuals $R_T$ and $R_{\partial
  T}$ of the residual
representation~\eqref{eq:residual_representation}.

\begin{figure}
  \begin{center}
    \includegraphics[width=0.6\textwidth]{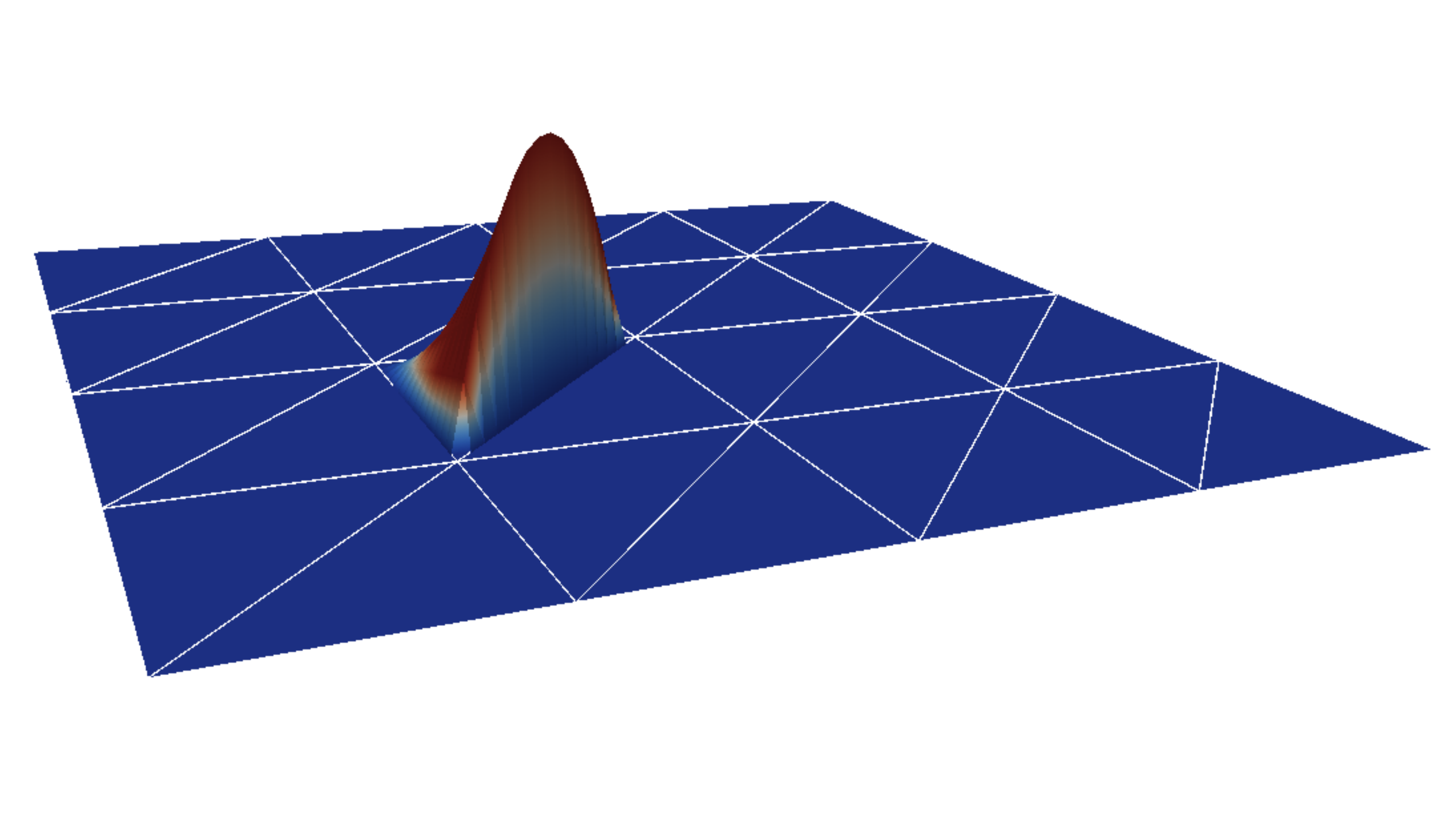}
    \includegraphics[width=0.3\textwidth]{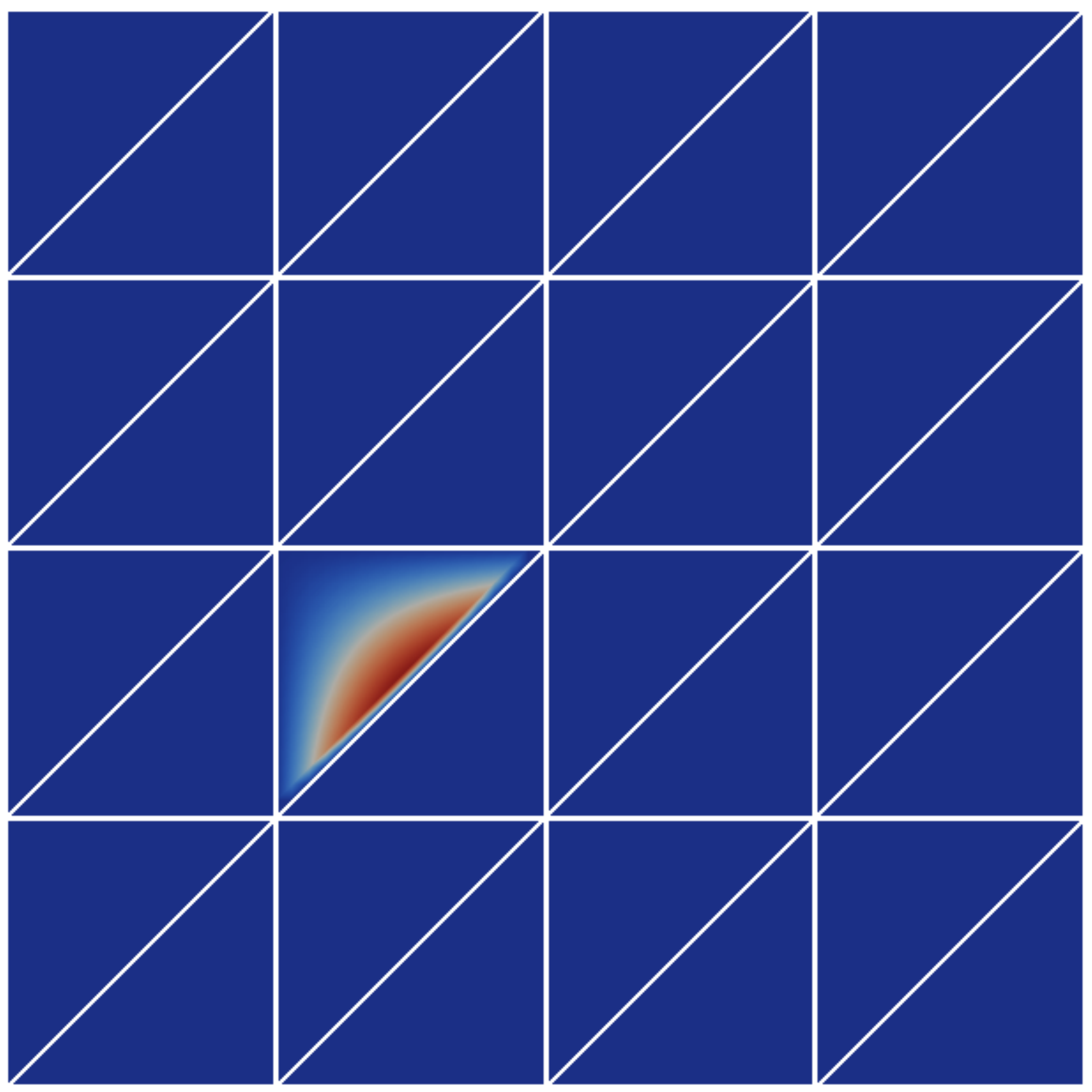}
    \caption{The cone function $\cone_{\facet}^T$.}
    \label{fig:edge_bubble}
  \end{center}
\end{figure}

One may thus compute the residual
representation~\eqref{eq:residual_representation} by solving a set of
local problems on each cell~$T$. First, one local problem for the cell
residual~$R_T$, and then local problems for the facet
residual~$R_{\partial T}$ restricted to each facet of~$T$. If the test
space is vector-valued, the local problems are solved for each scalar
component. We emphasize that the computation of the residual
representation~\eqref{eq:residual_representation} and thus the error
indicator~\eqref{eq:error_indicator} may be computed automatically
given only the variational problem~\eqref{eq:primal,linear} in terms
of the pair of bilinear and linear forms~$a$ and~$L$. In particular,
the derivation of the error indicators does not involve any manual
analysis.

We remark that the use of bubble and cone functions to localize the
weak residual is a standard technique for proving reliability and
efficiency for norm-based error estimates (see
e.g.~\cite{verfrth_review_1999}). The crucial observation here is that
this technique can also be used to automatically generate the residual
decomposition given only the weak residual. We also remark that the
local problems \eqref{eq:cell_residual} and~\eqref{eq:facet_residual}
are different from the local problems that were introduced
in~\cite{bank_posteriori_1985} to represent the cell and facet
residuals~$R_T$ and~$R_{\partial T}$ as a single residual.

\subsection{Solvability of the local problems}
\label{subsec:solvability}

To prove that the local problems~\eqref{eq:cell_residual}
and~\eqref{eq:facet_residual} uniquely determine the cell and facet
residuals, we recall the following result regarding bubble-weighted
$L^2$-norms. For a proof, we refer to~\cite[Theorems 2.2,
  2.4]{ainsworth_posteriori_2000}.
\begin{lem}
  \label{lem:bubble_weighted}
  Let $T$ be a $d$-simplex and let $b_T$ denote the bubble function on
  $T$. There exist positive constants $c$ and $C$, independent of $T$,
  such that
  \begin{equation}
    c ||\phi||_T^2 \leq \langle b_T \phi, \phi \rangle_T
    \leq C ||\phi||_T^2
  \end{equation}
  for all $\phi \in \Poly{p}(T)$.
\end{lem}

We may now prove the following theorem.
\begin{theorem}
  \label{th:solvability}
  If assumptions~\assumption{1}--\assumption{3} hold, then the cell
  and facet residuals of the residual
  representation~\eqref{eq:residual_representation} are uniquely
  determined by the local problems~\eqref{eq:cell_residual} and
  \eqref{eq:facet_residual}.
\end{theorem}
\begin{proof}
  Consider first the cell residual~$R_T$. Take $v = \bubble_T \phi_i$
  in~\eqref{eq:decomposition} for $i = 1, \dots, m$.  Since $v$
  vanishes on the cell boundary $\partial T$, we
  obtain~\eqref{eq:cell_residual}. By assumption, $R_T \in
  \Poly{p}(T)$ and is thus a solution of the local
  problem~\eqref{eq:cell_residual}. It follows from
  Lemma~\ref{lem:bubble_weighted}, that it is the unique solution.  We
  similarly see that the facet residual~$R_{\partial T}$ is a solution
  of the local problem~\eqref{eq:facet_residual} and uniqueness
  follows again from Lemma~\ref{lem:bubble_weighted}.
\end{proof}

In the cases where~\assumption{3} fails, such as if the variational
problem contains non-polynomial data, the local
problems~\eqref{eq:cell_residual} and~\eqref{eq:facet_residual}
uniquely determine the projections of $R_T$ and $R_{\partial
  T}|_{\facet}$ onto $\Poly{p}(T)$ and $\Poly{q}(\facet)$
respectively. The accuracy of the approximation may then be controlled
by the polynomial degrees $p$ and $q$. In the numerical examples
presented in Section~\ref{sec:numerics}, we let $p = q$ be determined
by the polynomial degree of the finite element space. We have not
observed any significant errors introduced by this approximation in
our numerical experiments.

Looking back at the special case of Poisson's
equation~\eqref{eq:errorindicator:poisson}, the cell and facet
residuals derived by hand are given by $R_T = f + \Delta u_h$ and
$R_{\partial T} = - \partial_n u_h$, respectively. We emphasize that,
by what is just shown, this indeed coincides with the representation
defined by~\eqref{eq:cell_residual} and~\eqref{eq:facet_residual} if
$f$ is polynomial.

%------------------------------------------------------------------------------
\section{Approximating the dual solution}
\label{sec:dual_approximation}

In order to evaluate the error
representation~\eqref{eq:error_representation} and to compute the
error indicators~\eqref{eq:error_indicator}, one must compute, or in
practice approximate, the solution~$z$ of the dual
problem~\eqref{eq:dual}. The natural discretization of~\eqref{eq:dual}
reads: find $z_h \in \dualtrialspace_h = \testspace_h$ such that
\begin{equation}
  \label{eq:dual,discrete}
  a^*(z_h, v) = \goal(v) \quad
  \foralls v \in \dualtestspace_h = \trialspace_{h, 0}.
\end{equation}
However, since the residual~$r$ vanishes on $\testspace_h$, $z_h$ is,
for the purpose of error estimation, highly unsuitable as an
approximation of the dual solution.

An immediate alternative is to solve the dual problem using a higher
order method. If the dual solution is sufficiently regular, a higher
order method would be expected to give a more accurate dual
approximation. It is observed in practice that a more accurate dual
approximation gives a better error
estimate~\cite{becker_optimal_2001}, although complete reliability
cannot be guaranteed~\cite{nochetto_safeguarded_2008}. Other
alternatives include approximation by hierarchic
techniques~\cite{ainsworth_posteriori_2000, bank_posteriori_1993} or
approximating the dual problem on a different mesh. In this work, we
suggest a new alternative based on solving~\eqref{eq:dual,discrete}
using the same mesh and polynomial order as the primal problem and
then extrapolating the computed solution~$z_h$ to a higher order
function space. This strategy can be compared to the higher order
interpolation procedure presented in~\cite{becker_optimal_2001} for
regular quadrilateral/hexahedral meshes. The strategy presented here
extends that of~\cite{becker_optimal_2001} however, as it can be
applied to almost arbitrary (admissible) simplicial tessellations.

To define the extrapolation procedure, let $V_h$ be a finite element
space on a tessellation~$\triang_h$ and let $\EVh \supset V_h$ be a
higher order finite element space on the same
tessellation~$\triang_h$. Furthermore, let $\{\phi^T_j\}_{j=1}^n$ be a
local basis for~$\EVh$ on~$T$ and let $\{\phi_j\}_{j=1}^N$ be the
corresponding global basis. For $v_h \in V_h$, we define the
extrapolation operator $E : V_h \rightarrow \EVh$ as described in
Algorithm~\ref{alg:extrapolation}. This algorithm computes the
extrapolation by fitting local polynomials to the finite element
function~$v_h$ on local patches. This yields a global multi-valued
function which is then averaged to obtain the extrapolation~$Ev_h$. We
illustrate the extrapolation algorithm in
Figure~\ref{fig:extrapolation} for a one-dimensional case.
\begin{algorithm}[t]
  \begin{enumerate}
  \item (Lifting) For each cell $T \in \triang_h$:
    \begin{enumerate}
    \item
      Define a patch of cells $\omega_T \supset T$ of sufficient size
      and let $\{\ell_i\}_{i=1}^m$ be the collection of degrees of
      freedom for $V_h$ on the patch. The size of the patch $\omega_T$
      should be such that the number of degrees of freedom $m$ is
      greater than or equal to the local dimension $n$ of $\EVh|_T$.
    \item
      Let $\{\phi^{\omega_T}_j\}_{j=1}^n$ be a smooth extension of
      $\{\phi^T_j\}_{j=1}^n$ to the patch $\omega_T$.
    \item
      Define $A_{ij} = \ell_i(\phi_j^{\omega_T})$ and $b_i =
      \ell_i(v_h)$ for $i = 1, \dots, m$, $j = 1, \dots, n$.
    \item
      Compute the least-squares approximation $\xi_T$ of the
      (overdetermined) $m \times n$ system $A \xi_T = b$.
    \end{enumerate}
  \item (Smoothing)
    \begin{enumerate}
    \item
      For each global degree of freedom $j$, let $X_j$ be the set of
      corresponding local expansion coefficients determined on each
      cell~$T$ by the local vector $\xi_T$. Define $\xi_j =
      \frac{1}{|X_j|} \sum_{x \in X_j} x$. We note that $|X_j| > 1$
      for degrees of freedom that are shared between cells.
    \item
      Define $E v_h = \sum_{j=1}^N \xi_j \phi_j$.
    \end{enumerate}
  \end{enumerate}
  \caption{Extrapolation}
  \label{alg:extrapolation}
\end{algorithm}
\begin{figure}[h]
  \begin{center}
    \includegraphics[width=0.7\textwidth]{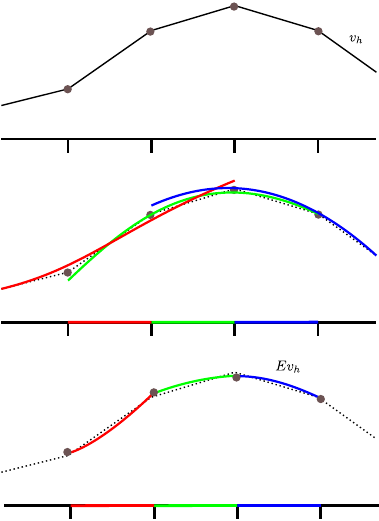}
    \caption{Extrapolation of a continuous piecewise linear
      function~$v_h$ to a continuous piecewise quadratic function $E
      v_h$.  The extrapolation is computed by first fitting a
      quadratic polynomial on each patch. In one dimension, each patch
      is a set of three intervals and each local quadratic polynomial
      is computed by solving an overdetermined $4 \times 3$ linear
      system. The continuous piecewise quadratic extrapolation $E v_h$
      is then computed by averaging at the end-points of each
      interval.}
    \label{fig:extrapolation}
  \end{center}
\end{figure}

Algorithm~\ref{alg:extrapolation} may be used to compute a higher
order approximation of the dual solution~$z$ as follows. First, we
compute an approximation~$z_h \in \testspace_h$ of the dual solution
by solving~\eqref{eq:dual,discrete}. We then compute the
extrapolation~$Ez_h \in \EVh$ where $\EVh$ is the finite element space
on $\triang_h$ obtained by increasing the polynomial degree by one. We
then estimate the error by
\begin{equation*}
  \eta \equiv |\goal(u) - \goal(u_h)| = |r(z)| \approx |r(E z_h)| \equiv \eta_h.
\end{equation*}

%------------------------------------------------------------------------------
\section{Extensions to nonlinear problems and goal functionals}
\label{sec:nonlinear}

We now turn to consider nonlinear variational problems and goal
functionals. We consider the following general nonlinear variational
problem: find $u \in \trialspace$ such that
\begin{equation} \label{eq:primal,nonlinear}
  F(u; v) = 0 \quad \foralls v \in \testspace.
\end{equation}
For a given nonlinear goal functional~$\goal : \trialspace \rightarrow
\R$, we define the following dual problem: find $z
\in \dualtrialspace$ such that
\begin{equation} \label{eq:dual,nonlinear}
  \overline{F'}^*(z, v) = \overline{\goal'}(v)
  \quad \foralls v \in \dualtestspace,
\end{equation}
where, as before, $\dualtestspace = \trialspace_0$ and
$\dualtrialspace = \testspace$. The bilinear form $\overline{F'}$ is
an appropriate average of the Fr\'echet derivative $F'(u; \delta u, v)
\equiv {\partial F(u; v) \over \partial u} \, \delta u$ of $F$,
\begin{equation}
  \label{eq:average}
  \overline{F'}(\cdot, \cdot) = \int_0^1 F'(su + (1-s)u_h; \cdot, \cdot) \ds.
\end{equation}
We note that by the chain rule, we have $\overline{F'}(u - u_h, \cdot)
= F(u; \cdot) - F(u_h; \cdot)$. The linear functional $\overline{M'}$
is defined similarly. Note that~\eqref{eq:dual,nonlinear} reduces
to~\eqref{eq:dual} in the linear case where $F(u; v) = a(u, v) -
L(v)$.

The following error representation now follows directly from the
definition of the dual problem:
\begin{equation*}
  \begin{split}
    \goal(u) - \goal(u_h)
    &= \overline{\goal'}(u - u_h)
    = \overline{F'}^*(z, u - u_h)
    = \overline{F'}(u - u_h, z) \\
    &= F(u; z) - F(u_h; z)
    = - F(u_h; z)
    \equiv r(z).
  \end{split}
\end{equation*}
We thus recover the error representation~\eqref{eq:error_representation}.

In practice, the exact solution~$u$ is not known and must be
approximated by the approximate solution~$u_h$; that is, the linear
operator $\overline{F'}$ is approximated by the derivative of~$F$
evaluated at $u = u_h$. The resulting linearization error may for the
sake of simplicity be neglected, as we shall in this exposition, but
doing so may reduce the accuracy (and reliability) of the computed
error estimates. For a further discussion on the issue of
linearization errors in the definition of the dual problem, we refer
to~\cite{becker_optimal_2001}.

It follows that the techniques described in
Section~\ref{sec:residual_representation}
and~\ref{sec:dual_approximation} directly apply to the residual $r$
and the dual approximation $z_h$ also for the nonlinear case.

%------------------------------------------------------------------------------
\section{A complete algorithm for automated goal-oriented error control}
\label{sec:adaptivity}

Based on the above discussion, we may now phrase the complete
algorithm for automated adaptive goal-oriented error control in
Algorithm~\ref{alg:adaptivity}.
\begin{algorithm}[H]
  Let $F : \trialspace \times \testspace \rightarrow \R$ be a given
  semilinear form, let $\goal : \trialspace \rightarrow \R$ be a given
  goal functional, and let $\tol > 0$ be a given tolerance.
  \begin{enumerate}
  \item
    Select an initial tessellation~$\triang_h$ of the domain~$\Omega$
    and construct the corresponding trial and test spaces
    $\trialspace_h \subset \trialspace$ and $\testspace_h
    \subset \testspace$ (for a given fixed finite element family and
    degree).
  \item
    Compute the finite element solution~$u_h \in \trialspace_h$ of the
    primal problem~\eqref{eq:primal,nonlinear} satisfying $F(u_h; v) =
    0$ for all $v \in \testspace_h$.
  \item
    Compute the finite element solution~$z_h \in \dualtrialspace_h$ of
    the dual problem~\eqref{eq:dual,nonlinear} satisfying
    $\overline{F'}^*(z, v) = \overline{\goal'}(v)$ for all $v
    \in \dualtestspace_h$.
  \item
    Extrapolate $z_h \mapsto E z_h$ using Algorithm~\ref{alg:extrapolation}.
  \item
    Evaluate the error estimate $\eta_h = |F(u_h; Ez_h)|$.
  \item
    If $|\eta_h| \leq \tol$, accept the solution~$u_h$ and
    break. (Stopping criterion)
  \item
    Compute the cell and facet residuals~$R_T$ and $R_{\partial T}$ of
    the residual representation~\eqref{eq:residual_representation} by
    solving the local problems~\eqref{eq:cell_residual}
    and~\eqref{eq:facet_residual}.
  \item
    Compute the error indicators \\
    $\eta_T
    = | \inner{R_T}{Ez_h - \pi_h Ez_h}_T
    + \average{\inner{R_{\partial T}}{Ez_h - \pi_h Ez_h}_{\partial T}}|$.
  \item
    Sort the error indicators in order of decreasing size and mark the
    first~$M$ cells for refinement where $M$ is the smallest number
    such that $\sum_{i=1}^M \eta_{T_i} \geq \alpha
    \sum_{T\in\triang_h} \eta_T$, for some choice of $\alpha \in (0,
    1]$. (D\"orfler marking~\cite{dorfler_convergent_1996})
  \item
    Refine all cells marked for refinement (and propagate refinement
    to avoid hanging nodes).
  \item
    Go back to step 2.
  \end{enumerate}
  \caption{Adaptive algorithm}
  \label{alg:adaptivity}
\end{algorithm}

%------------------------------------------------------------------------------

%\subsection{Implementation}
%\label{sec:implementation}

Algorithm~\ref{alg:adaptivity} has been implemented within the
\fenics{} project~\cite{LoggMardalEtAl2012a,logg_automatingfinite_2007, logg_dolfin_2009}, a collaborative project for the development of concepts
and software for automated solution of differential equations. The
implementation is freely available, and distributed as part of
\dolfin{} (version 0.9.11 and onwards). We discuss some of the
features of the implementation here and provide a simple use case.
More details of the implementation will be discussed in future
work~\cite{rognes_efficient_2012}.

For the specification of variational problems, the Python interface of
\dolfin{} accepts as input variational forms expressed in the form
language \ufl~\cite{Simula.SC.626}. Forms expressed in the \ufl{}
language are automatically passed to the FEniCS form
compiler~\ffc{}~\cite{kirby_compiler_2006,kirby_efficient_2007,LoggOlgaardEtAl2012a}
which generates efficient C++ code for finite element assembly of the
corresponding discrete operators. For a detailed discussion,
see~\cite{logg_dolfin_2009}. Stationary discrete variational problems
can be solved in DOLFIN by calling the \emp{solve} function accepting
as input a variational problem specified by a variational equation
expressed by two variational forms (defining the left- and right-hand
sides), the solution function \emp{u} and any boundary conditions
\emp{bcs}. Our implementation adds the possibility of solving such
problems adaptively with goal-oriented error control by adding a goal
functional \emp{M} and an error tolerance, say \emp{1e-6}:
\begin{python}
solve(a == L, u, bcs, tol=1.e-6, M=M) # Linear case
solve(F == 0, u, bcs, tol=1.e-6, M=M) # Nonlinear case
\end{python}
A simple complete example is listed in Figure~\ref{fig:samplecode}. A
number of optional parameters may be specified to control the behavior
of the adaptive algorithm, including the marking strategy and the
refinement fraction. The default marking strategy is D\"orfler
marking~\cite{dorfler_convergent_1996} with a refinement fraction of
$\alpha = 0.5$.

Internally, the adaptive algorithm relies on the capabilities of the
form language~\ufl{} for generating the dual problem, the local
problems for the cell and facet residuals, and the computation of
error indicators. As an illustration, we show here the code for
generating the bilinear form~$a^* = F'^*$ of the dual
problem~\eqref{eq:dual,nonlinear}:
\begin{python}
a_star = adjoint(derivative(F, u))
\end{python}
%% We also demonstrate the code for computing the error
%% indicators~$\eta_T$, which may be obtained by
%% assembling~\eqref{eq:error_indicator} tested against a piecewise
%% constant test function~\emp{v}. (Note that for a piecewise constant
%% function, the operation \emp{avg(v) = (v('+') + v('-'))/2} evaluates
%% to either \emp{v('+')/2} or \emp{v('-')/2} depending on which side of
%% the facet~$S$ is associated with the support of~$v$.)
%% \begin{python}
%% Ez_h = extrapolate(z_h, EV_h)
%% w = Ez_h - interpolate(Ez_h, V_h)
%% indicator_form = v*inner(R_T, w)*dx
%%   + avg(v)*(inner(R_dT('+'), w('+')) + inner(R_dT('-'), w('-')))*dS
%%   + v*inner(R_dT, w)*ds
%% indicators = assemble(indicator_form)
%% \end{python}

\begin{figure}
\begin{python}
from dolfin import *

mesh = UnitSquare(4, 4)
V = FunctionSpace(mesh, "CG", 1)
u = Function(V)
v = TestFunction(V)
f = Constant(1.0)

F = inner((1 + u**2)*grad(u), grad(v))*dx - f*v*dx
bc = DirichletBC(V, 0.0, "near(x[0], 0.0)")

M = u*dx
solve(F == 0, u, bc, tol=1.e-3, M=M)
\end{python}
\caption{Complete code for the automated adaptive solution of a
  nonlinear Poisson-like problem on the unit square with $f = 1.0$,
  homogeneous Dirichlet boundary conditions on the left boundary and
  homogeneous Neumann conditions on the remaining boundary, with goal
  functional $\mathcal{M} = \int_{\Omega} u \dx$.}
\label{fig:samplecode}
\end{figure}

%------------------------------------------------------------------------------
\section{Numerical examples}
\label{sec:numerics}

In this section, we aim to investigate the performance of the
automated algorithm. Since the theoretical properties of the proposed
extrapolation procedure are largely unknown, the investigation here
focuses on the quality of the error estimate and the sum of the error
indicators on adaptively refined meshes. The total computational
efficiency of the automated adaptive algorithm will be investigated in
later works~\cite{rognes_efficient_2012}.

We present three numerical examples from three different application
areas, aiming to illustrate different characteristics and varying
levels of complexity. We begin by considering a basic example: a
standard discretization of the Poisson equation, and evaluate the
quality of the error estimates; the results show that the algorithm
gives error indicators close to the optimal value of one. The second
example is a discretization of a weakly symmetric formulation for
linear elasticity. This discretization involves a nontrivial finite
element space, namely, a mixed finite element space consisting of
multiple Brezzi--Douglas--Marini elements, and multiple discontinuous
and continuous elements. As far as the authors are aware, this is the
first demonstration of goal-oriented error control for the
discretization presented. The results show that the algorithm produces
error estimates of optimal quality also for this far more complicated
case. Finally, we consider a nonlinear, nonsmooth example of
wide-spread use: a mixed discretization of the incompressible
Navier--Stokes equations and evaluate both the quality of the error
estimates and the performance of the adaptive algorithm.

\subsection{The Poisson equation}
\label{subsec:poisson}

We begin by considering the Poisson equation:
\begin{subequations}
  \label{eq:poisson:strong}
  \begin{align}
    - \Delta u &= f \text{ in } \Omega, \\
    u &= 0  \text{ on } \partial \Omega_D, \\
    \partial_n u &= g \text{ on } \partial \Omega_N.
  \end{align}
\end{subequations}
The standard variational formulation of~\eqref{eq:poisson:strong} fits
the framework of Section~\ref{sec:linear} with $\trialspace
= \testspace = H^1_{0, \partial \Omega_D}(\Omega)$ and
\begin{subequations}
  \label{eq:poisson:weak}
  \begin{align}
    a(u, v) &= \langle \Grad u, \Grad v \rangle, \\
    L(v) &= \langle f, v \rangle + \langle g, v \rangle_{\partial \Omega_N}.
  \end{align}
\end{subequations}
We consider the discretization of~\eqref{eq:poisson:weak} using the
space of continuous piecewise linear polynomials that satisfy the
essential boundary condition for $\trialspace_h = \testspace_h$.

As a test case, we consider a three-dimensional L-shaped domain,
\begin{equation*}
  \Omega =
  \left ( (-1, 1) \times (-1, 1) \setminus (-1, 0) \times (-1, 0) \right )
  \times (-1, 0),
\end{equation*}
with Dirichlet boundary $\partial \Omega_D = \{ (x, y, z) \, : \, x =
1 \text{ or } y = 1 \}$ and Neumann boundary $\partial \Omega_N =
\partial \Omega \setminus \partial \Omega_D$. Let $f(x, y, z) = - 2 (x
- 1)$ and let $g = G \cdot n$ with $G(x, y, z) = \left ( (y - 1)^2,
2(x - 1)(y - 1), 0 \right)$. The exact solution is then given by
\begin{equation}
  \label{eq:poisson:solution}
  u(x, y, z) = (x - 1) (y - 1)^2.
\end{equation}

As a goal functional, we take the average value of the solution on the
left boundary $\Gamma = \{ (x, y, z) \, : \, x = -1 \}$; that is,
\begin{equation}
  \goal(u) = \int_{\Gamma} u \ds.
\end{equation}
It follows that the exact value of the goal functional is $\goal(u) =
-2/3$.

Figure~\ref{fig:poisson} shows errors $\eta$, error estimates
$\eta_h$, the sum of the error indicators $\sum_T \eta_T$, and
efficiency indices $\eta_h/\eta$ and $\sum_T \eta_T/\eta$ for a series
of adaptively (and automatically) refined meshes. We first note that
the error estimate $\eta_h$ is very close to the error $\eta$. On the
coarsest mesh, the efficiency index is $\eta_h/\eta \approx 0.89$ and
as the mesh is refined, the efficiency index quickly approaches
$\eta_h/\eta \approx 1$. We further note that the sum of the error
indicators tends to overestimate the error, but only by a small
constant factor. This demonstrates that the automatically computed
error indicators are good indicators for refinement. We emphasize that
since the error indicators are not used as a stopping criterion for
the adaptive refinement, it is not important that they sum up to the
error.
\begin{figure}[ht]
  \begin{center}
    \subfigure[Errors]{
      \includegraphics[width=0.98\textwidth]{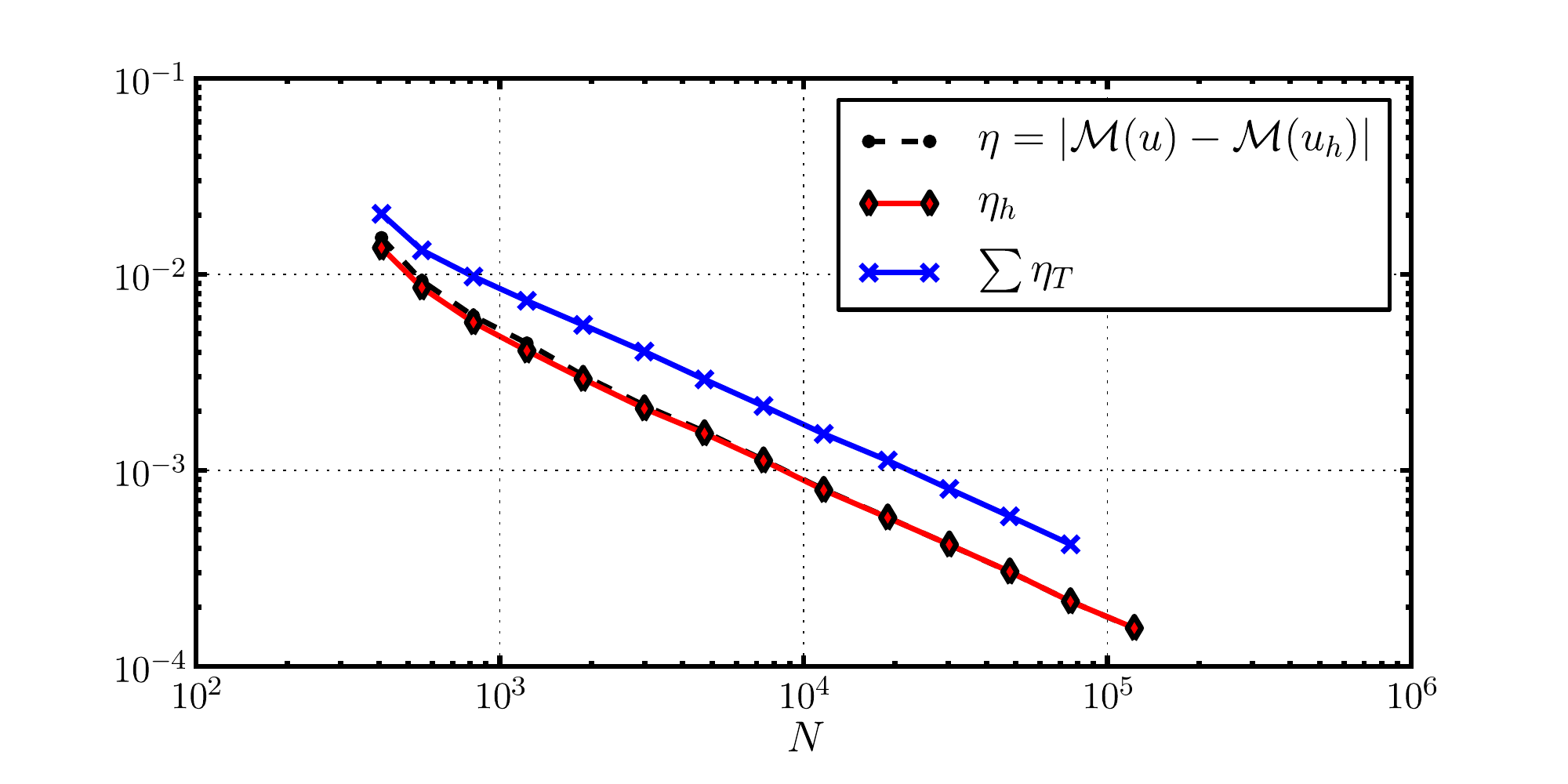}
      \label{fig:poisson:error}
    } \subfigure[Efficiency indices]{
      \includegraphics[width=0.98\textwidth]{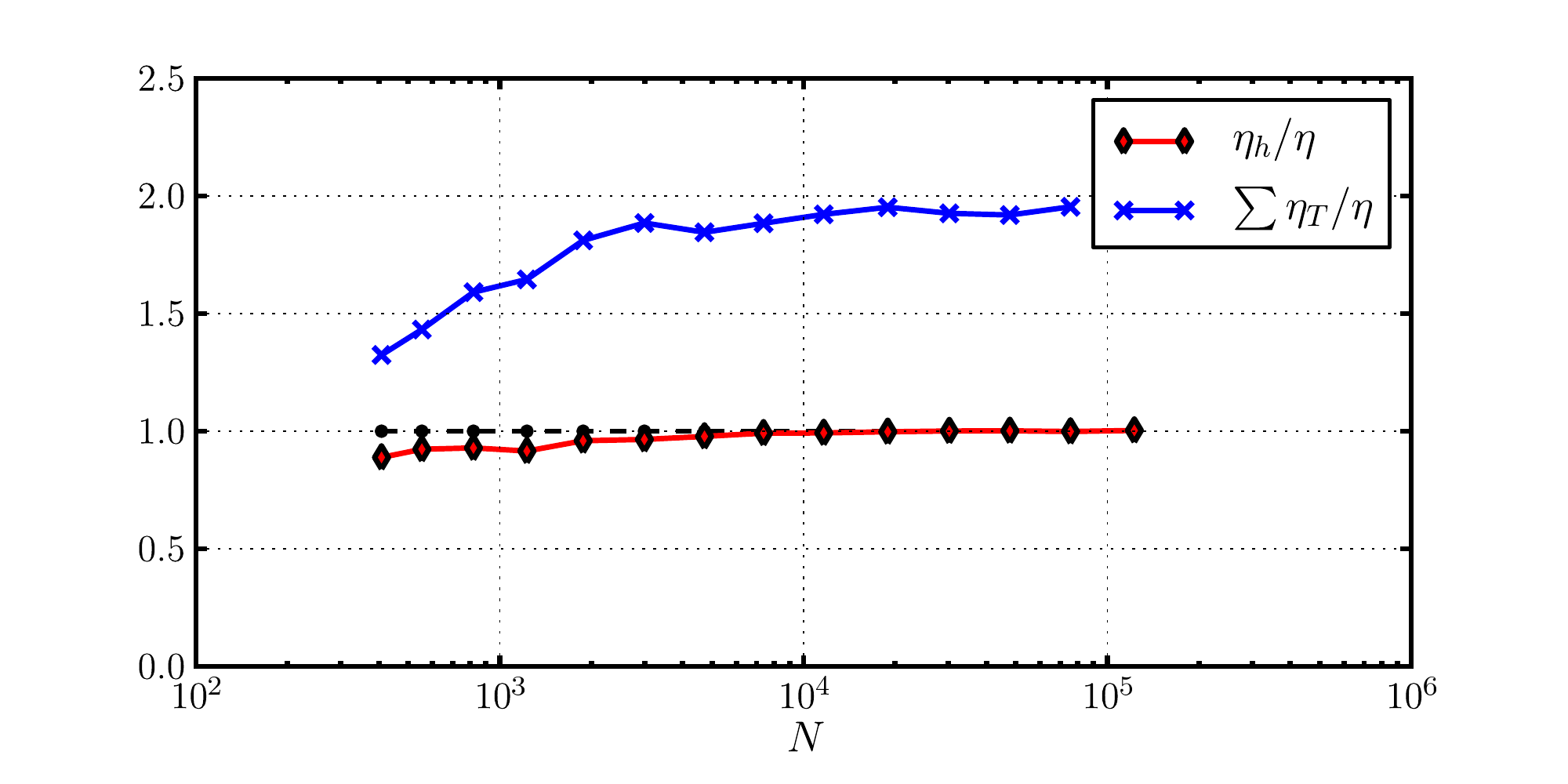}
      \label{fig:poisson:index}
    }
    \caption{Errors, error estimates, and summed error indicators
      (top) and efficiency indices (bottom) versus the number of
      degrees of freedom $N$ for adaptively refined meshes for the
      Poisson problem.  Note the excellent agreement between the error
      $\eta$ (dashed black curve) and the error estimate $\eta_h$
      (solid red curve), as well as the convergence of the efficiency
      index $\eta_h / \eta$ towards~$1$.}
\label{fig:poisson} \end{center}
\end{figure}

\subsection{Weakly symmetric linear elasticity}

As a more challenging test problem, we consider a three-field
formulation for linear isotropic elasticity enforcing the symmetry of
the stress tensor weakly. This gives rise to a mixed formulation that
involves $H(\Div)$- and $L^2$-conforming spaces. For a domain $\Omega
\subset \R^2$, the unknowns are the stress tensor $\sigma \in H(\Div,
\Omega; \mathbb{R}^{2 \times 2}) $, the displacement $u \in
L^2(\Omega; \R^{2})$, and the rotation $\gamma \in L^2(\Omega)$. The
bilinear and linear forms read
\begin{subequations} \label{eq:elasticity}
  \begin{align}
  a((\sigma, u, \gamma), (\tau, v, \eta))
  &= \inner{A \sigma}{\tau}
  + \inner{\Div \sigma}{v}
  + \inner{u}{\Div \tau}
  + \inner{\sigma}{\eta}
  + \inner{\gamma}{\tau}, \\
  L((\tau, v, \eta))
  &= \inner{g}{v}
  + \inner{u_0}{\tau \cdot n}_{\partial \Omega}.
  \end{align}
\end{subequations}
Here, $g$ is a given body force, $u_0$ is a prescribed boundary
displacement field, and $A$ is the compliance tensor. For isotropic,
homogeneous elastic materials with shear modulus $\mu$ and stiffness
$\lambda$, the action of $A$ reduces to
\begin{equation}
  \label{eq:A:isotropic}
  A \sigma = \frac{1}{2\mu} \left ( \sigma -
  \frac{\lambda}{2 (\mu + \lambda)} (\tr \, \sigma)  I \right ).
\end{equation}

We consider the discretization of these equations by a mixed finite
element space $\trialspace_h = \testspace_h$ consisting of the tensor
fields composed of two first-order Brezzi--Douglas--Marini elements
for the stress tensor, piecewise constant vector fields for the
displacement, and continuous piecewise linears for the
rotation~\cite{falk_finite_2008, farhloul_dual_1997}.

We consider the domain $\Omega = (0, 1) \times (0, 1)$ and the exact
solution $u(x, y) = (x y \sin(\pi y), 0)$ for $\mu = 1$ and $\lambda =
100$, and insert
\begin{equation*}
  g = \Div A^{-1} \varepsilon (u)
    =
  \left (
  \begin{array}{c}
    \pi \mu (2 x \cos(\pi y) - \pi x y \sin(\pi y)) \\
    \mu ( \pi y \cos(\pi y) + \sin(\pi y) )
    + \lambda (\pi y \cos(\pi y) + \sin (\pi y) )
  \end{array}
  \right ).
\end{equation*}
As a goal functional, we take a weighted measure of the average shear
stress on the right boundary,
\begin{equation*}
  \goal((\sigma, u, \gamma)) = \int_{\Gamma} \sigma \cdot n
  \cdot (\psi, 0) \ds
  \approx -0.06029761071,
\end{equation*}
where $\Gamma = \{ (x, y) \, : \, x = 1 \}$ and $\psi = y (y - 1)$.

The resulting errors, error estimates, error indicators, and
efficiency indices are plotted in Figure~\ref{fig:elasticity}. Again,
we note that the error estimate~$\eta_h$ is very close to the actual
error~$\eta$. We also note the good performance of the error
indicators that overestimate the error by around a factor of $2-4$.
This is remarkable, considering that that the error estimate and error
indicators are derived automatically for a non-trivial mixed
formulation and involve automatic extrapolation of the dual solution
from a mixed $[\mathrm{BDM}_1]^2 \times \mathrm{DG}_0 \times P_1$
space to a mixed $[\mathrm{BDM}_2]^2 \times \mathrm{DG}_1 \times P_2$
space. As far as the authors are aware, this is the first
demonstration of goal-oriented error control for this discretization
of the formulation~\eqref{eq:elasticity}.

\begin{figure}[ht]
  \begin{center}
    \subfigure[Errors]{
      \includegraphics[width=0.98\textwidth]{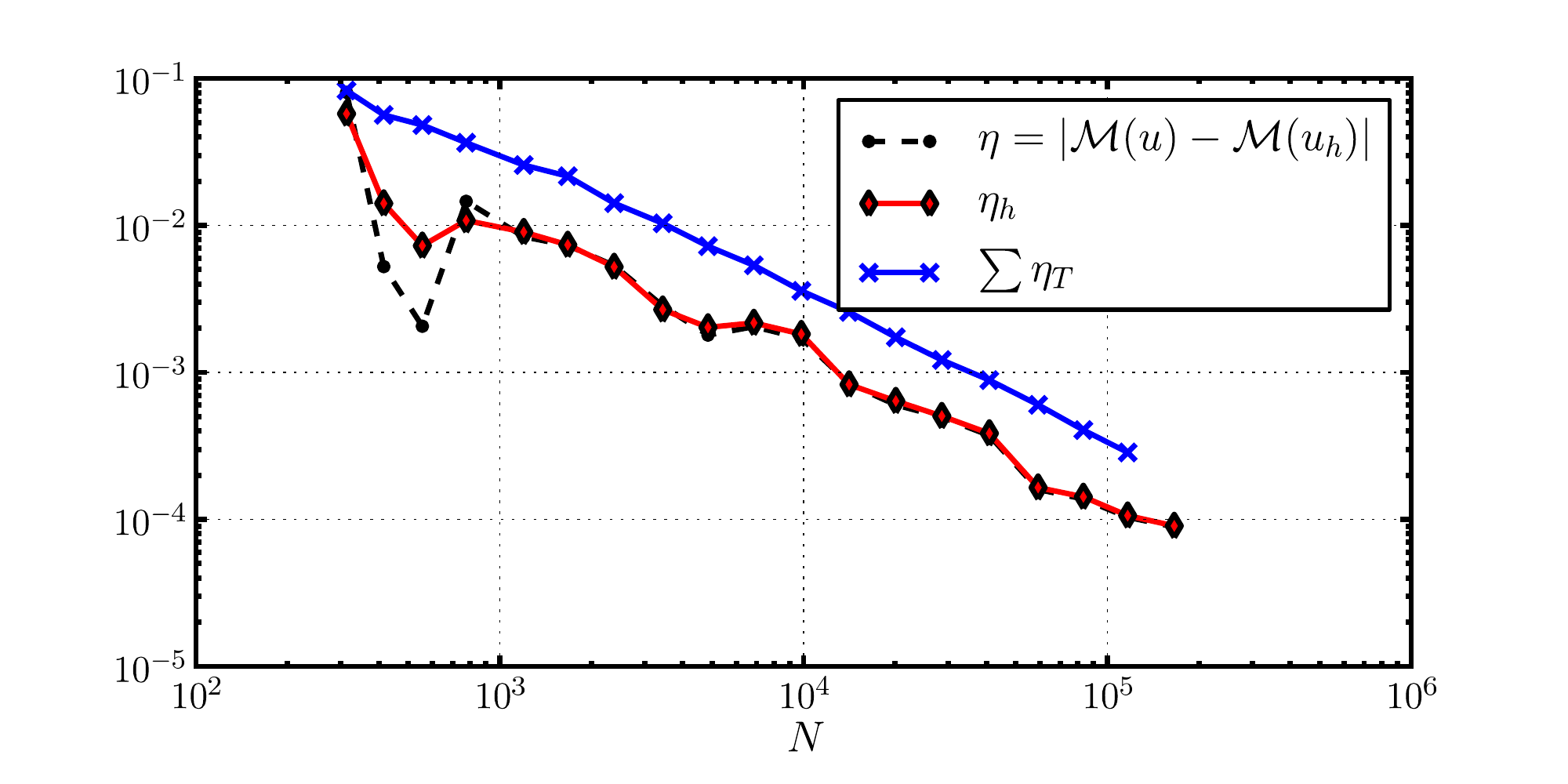}
      \label{fig:elasticity:error}
    } \subfigure[Efficiency indices]{
      \includegraphics[width=0.98\textwidth]{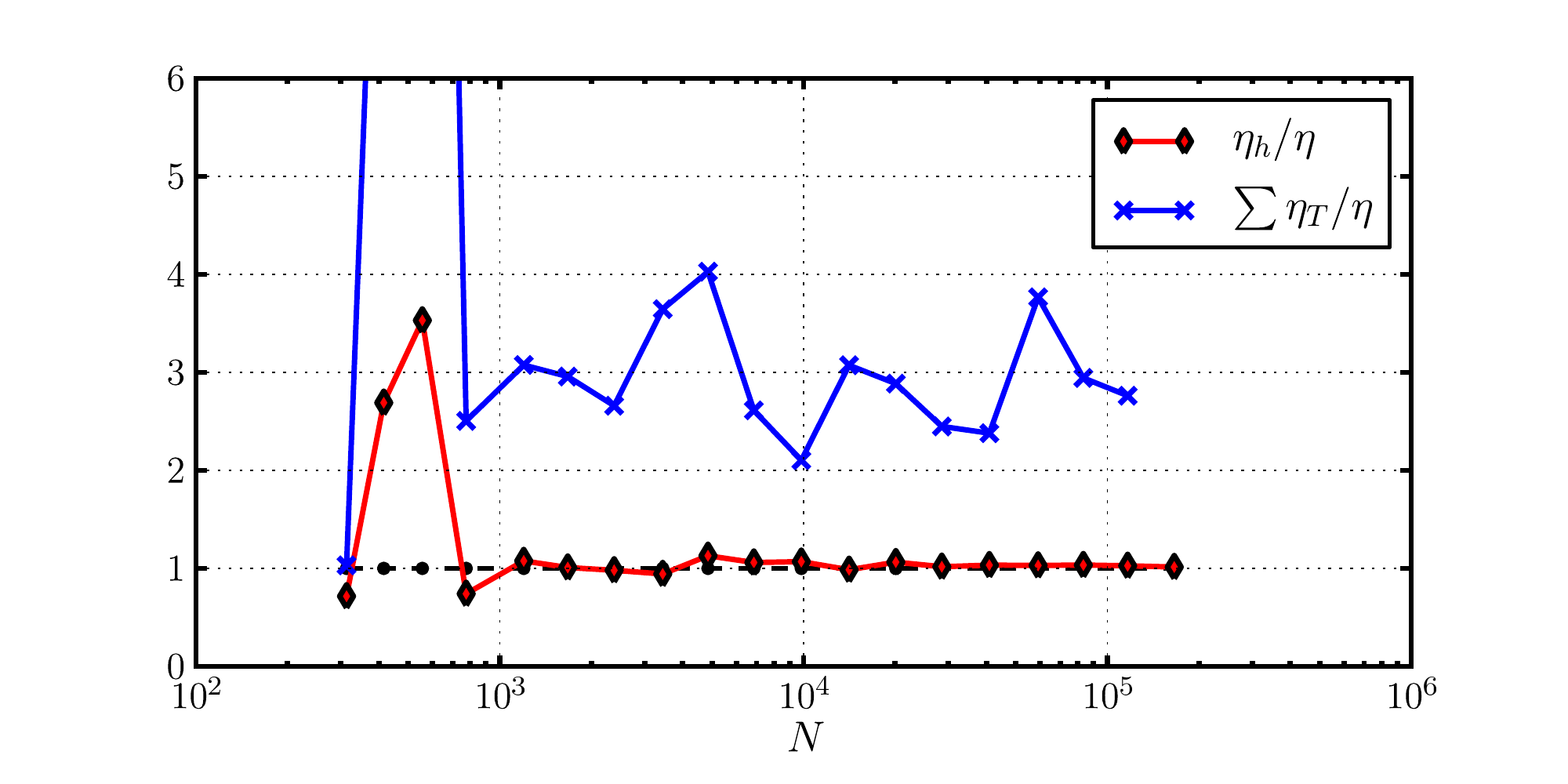}
      \label{fig:elasticity:index}
    }
    \caption{Errors, error estimates, and summed error indicators
      (top) and efficiency indices (bottom) versus the number of
      degrees of freedom $N$ for adaptively refined meshes for the
      mixed elasticity problem.  Note the excellent agreement between
      the error $\eta$ (dashed black curve) and the error estimate
      $\eta_h$ (solid red curve), as well as the convergence of the
      efficiency index $\eta_h / \eta$ towards~$1$.}
    \label{fig:elasticity}
  \end{center}
\end{figure}

\subsection{The stationary incompressible Navier--Stokes equations}

Finally, we consider a stationary pressure-driven Navier--Stokes flow
in a two-dimensional channel with an obstacle. We let $\Omega =
\Omega_C \backslash \Omega_O$, where $\Omega_C = (0, 4) \times (0, 1)$
and $\Omega_O = (1.4, 1.6) \times (0, 0.5)$. We let $\Omega_N = \{ (x,
y) \in \partial \Omega, x = 0 \text{ or } x = 4\}$ denote the Neumann
(inflow/outflow) boundary and let $\Omega_D = \Omega \setminus
\Omega_N$ denote the Dirichlet (no-slip) boundary.

We consider the following nonlinear variational problem for the
solution of the stationary incompressible Navier--Stokes equations:
find $(u, p) \in \trialspace$ such that \[F((u, p); (v, q)) = 0\] for
all $(v, q) \in \testspace$, where
\begin{equation*}
  \label{eq:navier-stokes}
  F((u, p); (v, q)) =
  \nu \inner{\Grad u}{\Grad v}
  + \inner{\Grad u \cdot u}{v}
  - \inner{p}{\Div v}
  + \inner{\Div u}{q}
  + \inner{\bar{p}n}{v}_{\partial \Omega_N}.
\end{equation*}
Here, $\bar{p}$ is a given boundary condition at the inflow/outflow
boundary.

The trial and test spaces are given by $\trialspace = \testspace =
H^1_{0, \partial \Omega_D}(\Omega; \R^2) \times L^2(\Omega)$. We let
the (kinematic) viscosity be $\nu = 0.02$ and take $\bar{p} = 1$ at $x
= 0$ (inflow) and $\bar{p} = 0$ at $x = 4$ (outflow). The quantity of
interest is the outflux at $x = 4$,
\begin{equation*}
  \goal(u, p) = \int_{x = 4} u \cdot n \ds
  \approx 0.40863917.
\end{equation*}
The system is discretized using a Taylor--Hood elements; that is, the
velocity space is discretized using continuous piecewise quadratic
vector fields and the pressure space is discretized using continuous
piecewise linears. The nonlinear system is solved using a standard
Newton iteration.

The results for this case are shown in Figure~\ref{fig:ns}. As seen in
this figure, the error estimate is not as accurate as for the two
previous test cases. The efficiency index oscillates in the range
$0.2-1.0$.  This is not surprising, considering that (i) a
linearization error is introduced when linearizing the dual problem
around the computed approximate solution $u_h$, rather than computing
the average~\eqref{eq:average}, and (ii) both the primal and dual
problems exhibit singularities at the reentrant corners making the
higher-order extrapolation procedure suboptimal for approximating the
exact dual solution. Still, we obtain reasonably good error estimates
and error indicators. Furthermore, the adaptive algorithm performs
very well when comparing the convergence obtained with the adaptively
refined sequence of meshes to that of uniform refinement,
cf.~Figure~\ref{fig:convergencevsuniform}. The final mesh is shown
in~Figure~\ref{fig:navier-stokes:mesh}, and we note that it is heavily
refined in the vicinity of the reentrant corners.
\begin{figure}[ht]
  \begin{center}
    \subfigure[Errors]{
      \includegraphics[width=0.98\textwidth]{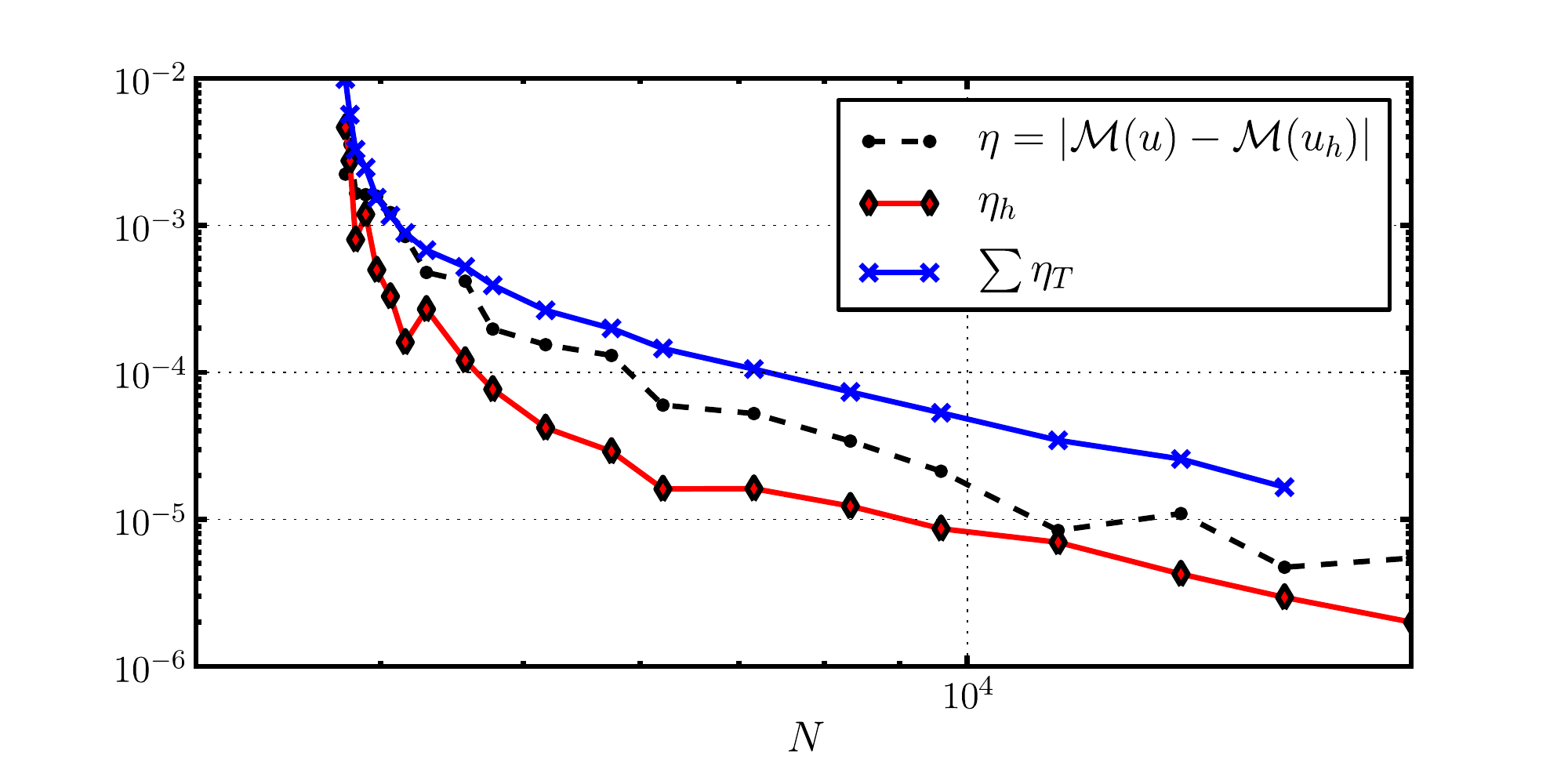}
      \label{fig:ns:error}
    } \subfigure[Efficiency indices]{
      \includegraphics[width=0.98\textwidth]{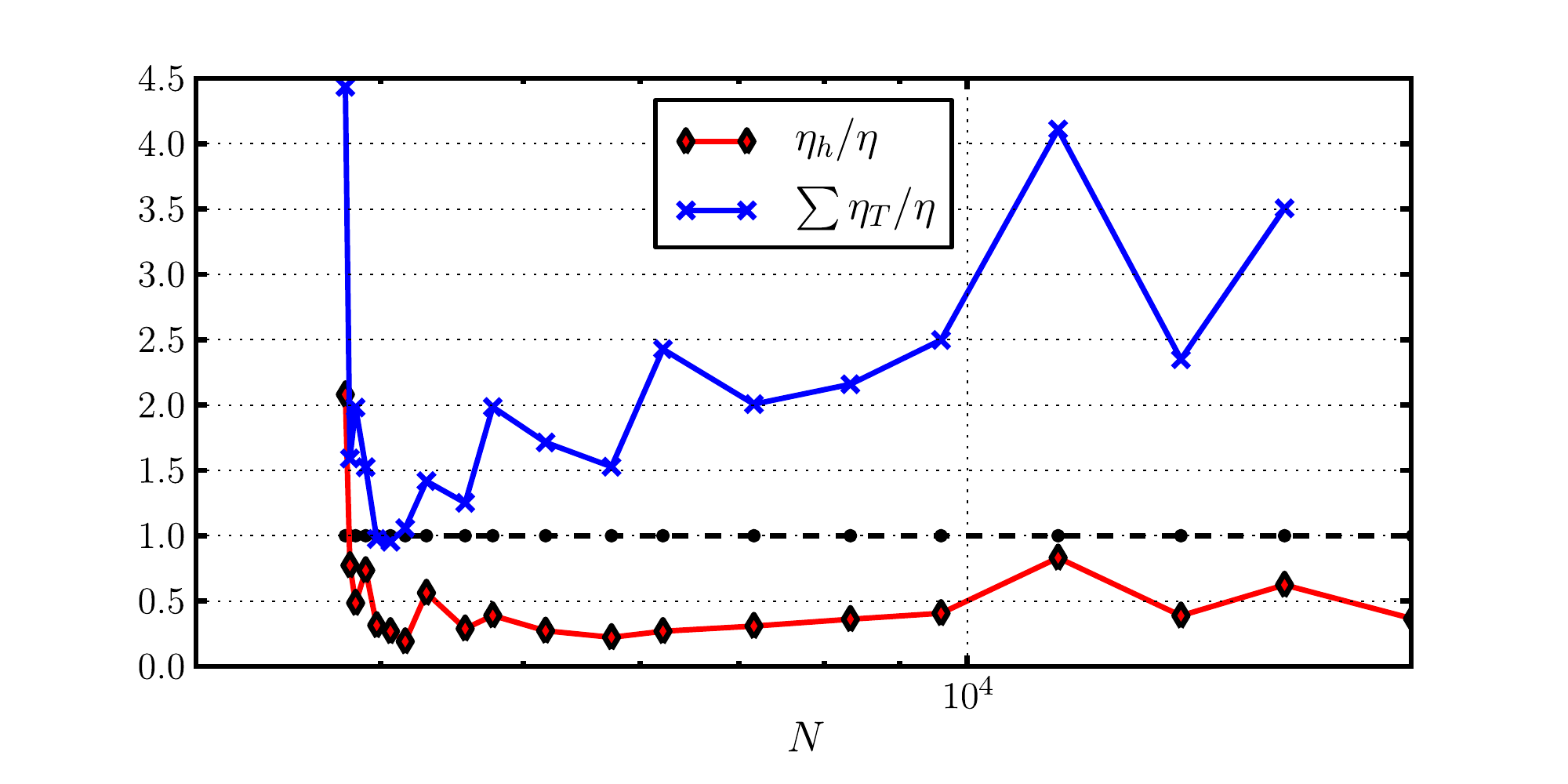}
      \label{fig:ns:index}
    }
    \caption{Errors, error estimates, and summed error indicators
      (top) and efficiency indices (bottom) versus the number of
      degrees of freedom $N$ for adaptively refined meshes for the
      Navier--Stokes problem. This is a detail of the full convergence
      plot shown in Figure~\ref{fig:convergencevsuniform}, where the
      convergence of the adaptive algorithm is contrasted to the
      convergence obtained with uniform refinement.}  \label{fig:ns}
  \end{center}
\end{figure}

\begin{figure}[ht]
  \begin{center}
    \includegraphics[width=0.98\textwidth]{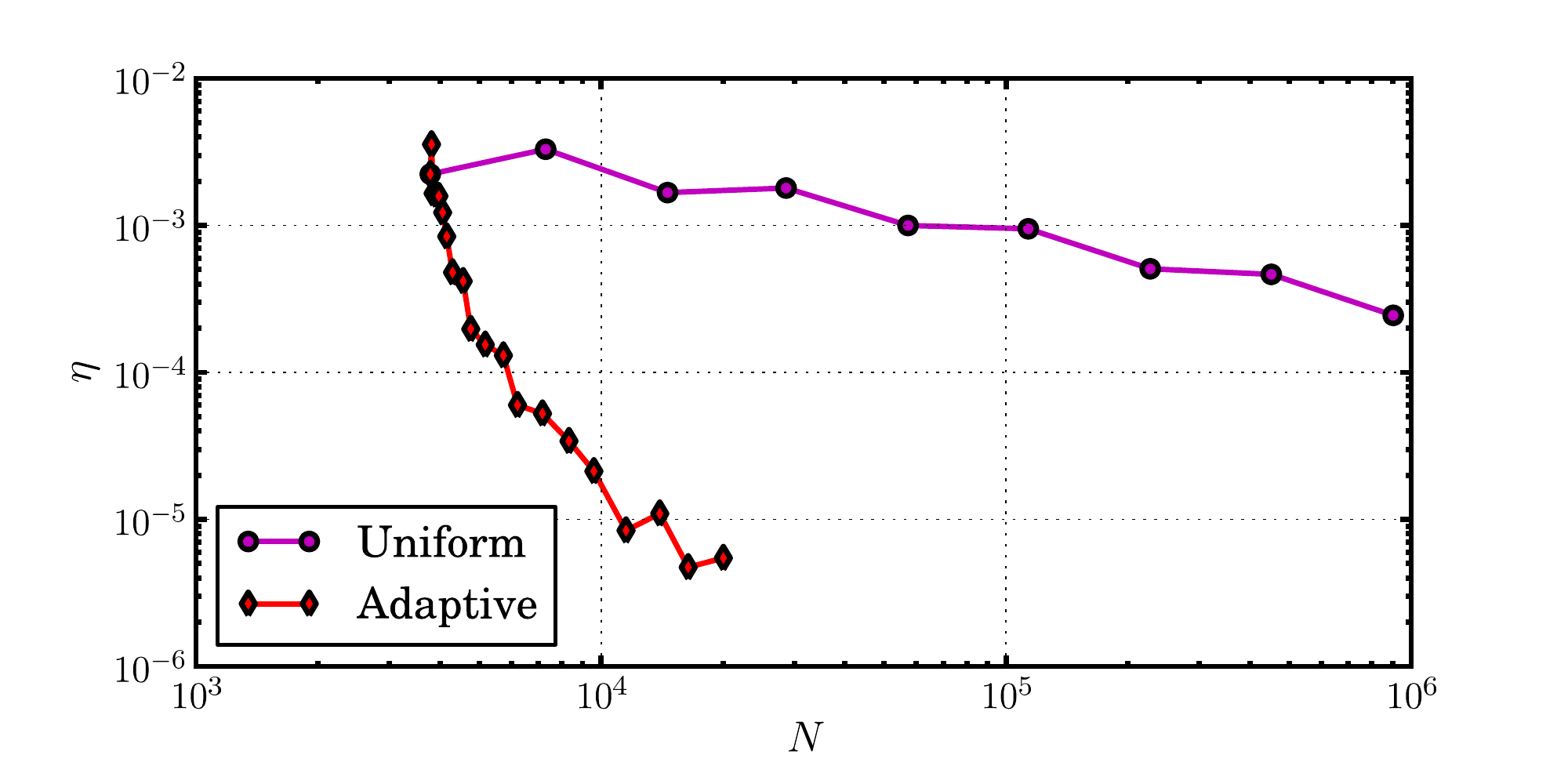}
    \caption{Convergence for adaptively and uniformly refined meshes
      for the Navier--Stokes problem. Adaptive refinement outperforms
      uniform refinement by 1--2 orders of magnitude.}
    \label{fig:convergencevsuniform}
  \end{center}
\end{figure}

\begin{figure}
  \begin{center}
    \includegraphics[width=0.9\textwidth]{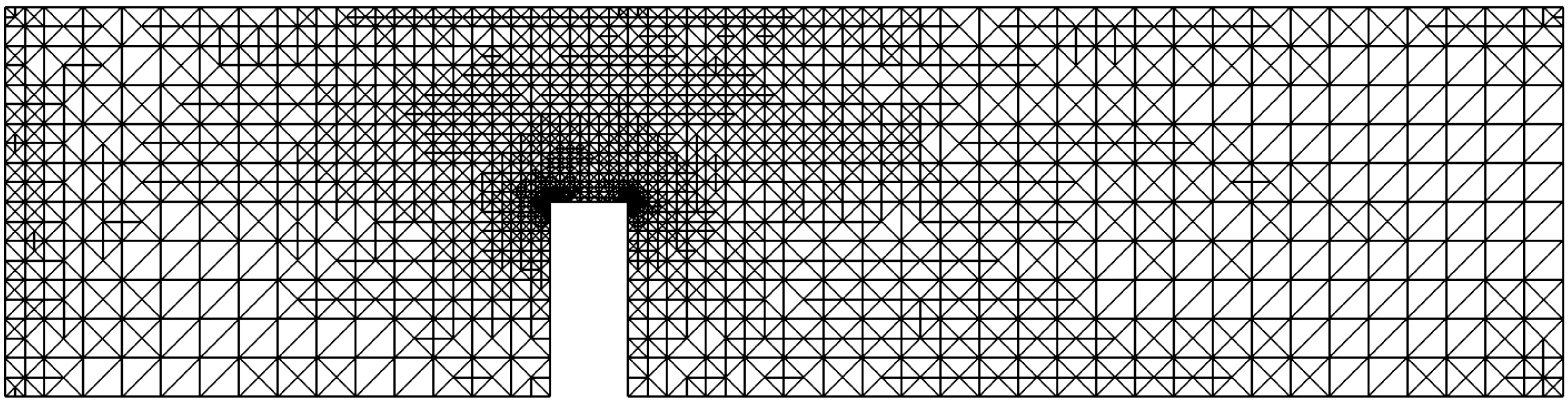}
    \caption{Final mesh for the Navier--Stokes problem.}
    \label{fig:navier-stokes:mesh}
  \end{center}
\end{figure}

%------------------------------------------------------------------------------
\section{Conclusions}
\label{sec:conclusion}

We have demonstrated a new strategy for automated, adaptive solution
of finite element variational problems. The strategy is implemented
and freely available as part of the DOLFIN finite element
library~\cite{logg_dolfin_2009}; accessible both through the Python
and C++ interfaces.

The strategy and its implementation are currently limited to
stationary nonlinear variational problems. Another limitation is the
restriction to conforming finite element discretizations. These are
both issues that we plan to consider in future extensions of this
work. Additionally, we have assumed that the dual problem is
well-posed. This assumption may fail in cases where the primal problem
has been stabilized by the introduction of additional terms; the
adjoint (linearized) dual problem is then not necessarily well-posed.
Automated error control for such formulations is an interesting topic
for further research, but is beyond the scope of the present work.

Although the implementation has been tested on a number of model
problems with convincing results, the effect of the linearization
error (approximating $u \approx u_h$ in~\eqref{eq:average}) is
unknown. As a consequence, the computed error estimates typically
underestimate the error for nonlinear problems. The effect of the
linearization error and its proper treatment remains an open (and
fundamental) question. Also, the extrapolation algorithm proposed and
numerically tested here should be examined from a theoretical
viewpoint.

We remark that the techniques described in this paper could also be
used for norm-based error estimation. \emph{A~posteriori} error
estimates for energy or other Sobolev norms typically rely on
computing appropriately weighted norms of cell and averaged facet
residuals. Hence, the strategy described here provides a
starting-point for the automatic generation of norm-based error
estimators.

\bibliographystyle{siam}
\bibliography{bibliography}

\end{document}